\pdfoutput=1
\documentclass[10pt]{article} % Journal requested 10pt article style

\usepackage{amsfonts}
\usepackage{amsthm}
\usepackage{amsmath}
\usepackage{amssymb}

\usepackage{stackengine}
\usepackage{enumerate}
\usepackage{mathtools}
\usepackage{tikz}
\usepackage{tikz-cd}
\usepackage[all]{xy}
\usepackage{graphicx}
\usepackage{adjustbox}
\usepackage{extpfeil}
\usepackage{comment}
\usepackage{mathabx}
\usepackage{float}
\usepackage[labelformat=empty]{caption}
\usepackage[toc]{appendix}
\usepackage{resizegather}
\usepackage{dsfont}
\usetikzlibrary{matrix}
\usepackage[normalem]{ulem}

\usepackage{authblk}

\usepackage[activate={true, nocompatibility}, final, tracking=true, kerning=true, spacing=true, factor=1100, stretch=10, shrink=10]{microtype}
\microtypecontext{spacing=nonfrench}

\usepackage{hyperref}
\hypersetup{pdfborder=0 0 0} 
\usepackage{cleveref}

\theoremstyle{plain}
\newtheorem{thm}{Theorem}[section]
\newtheorem{coroll}[thm]{Corollary}
\newtheorem{defn}[thm]{Definition}
\newtheorem{lemma}[thm]{Lemma}

\newtheorem{prop}[thm]{Proposition}

\newtheorem{remark}[thm]{Remark}

\newtheorem{introthm}{Theorem}
\newtheorem{notn}[thm]{Notation}

\newcommand\hollowslash{\setbox0=\hbox{/}\def\holwd{3pt}%
  \stackengine{-.3pt}{/}{\rlap{\kern-1pt\rule{\holwd}{.4pt}}}{O}{r}{F}{F}{S}%
  \kern\dimexpr\holwd-\wd0-.2pt\relax%
  \stackengine{-.4pt}{/}{\llap{\rule{\holwd}{.4pt}\kern-1pt}}{U}{l}{F}{F}{S}%
}

\tikzset{
  symbol/.style={
    draw=none,
    every to/.append style={
      edge node={node [sloped, allow upside down, auto=false]{$#1$}}}
  }
}

\DeclareMathOperator{\Sym}{Sym}

\newcommand{\fg}{\mathfrak{g}}
\newcommand{\ft}{\mathfrak{t}}
\newcommand{\fc}{\mathfrak{c}}
\newcommand{\fb}{\mathfrak{b}}
\newcommand{\fC}{\mathfrak{C}}

\newcommand{\cO}{\mathcal{O}}

\newcommand{\cD}{\mathcal{D}}

\newcommand{\cM}{\mathcal{M}}

\newcommand{\cH}{\mathcal{H}}

\newcommand{\cL}{\mathcal{L}}

\newcommand{\bC}{\mathbb{C}}
\newcommand{\bG}{\mathbb{G}}
\newcommand{\bF}{\mathbb{F}}

\newcommand{\bZ}{\mathbb{Z}}

\newcommand{\on}{\operatorname}

\newcommand{\Spec}{\on{Spec}}

\newcommand{\Bun}{ \on{Bun} } 

\newcommand{\GL}{\on{GL}}

\usepackage{color}

\newcommand{\stHig}{\cM_{Dol}}
\newcommand{\schHig}{\mathbb{M}_{Dol}}
\newcommand{\stdR}{\cM_{dR}}
\newcommand{\schdR}{\mathbb{M}_{dR}}
\newcommand{\stP}{\mathcal{P}}
\newcommand{\schP}{\mathbb{P}}
\newcommand{\stsp}{\cH}
\newcommand{\schsp}{\mathbb{H}}
\newcommand{\hbas}{A}
\newcommand{\wt}{\widetilde}
\newcommand{\oql}{\overline{\mathbb{Q}}_{\ell}}

\newcommand{\pcs}{ \,^{\mathfrak p}\!{\mathcal H}   }

\newcommand{\IC}{\mathcal{I}\mathcal{C}}

\newcommand{\git}{\mathbin{
  \mathchoice{/\mkern-6mu/}% \displaystyle
    {/\mkern-6mu/}% \textstyle
    {/\mkern-5mu/}% \scriptstyle
    {/\mkern-5mu/}}}% \scriptscriptstyle

\begin{document}
\title{Semistable Non Abelian Hodge theorem in positive characteristic}

\author[1]{Andres Fernandez Herrero}
\author[2]{Siqing Zhang}

% \affil command handles the address formatting
\affil[1]{Department of Mathematics, University of Pennsylvania, 209 South 33rd Street, Philadelphia, PA 19104, USA \\ \texttt{andresfh@sas.upenn.edu}}

\affil[2]{Department of Mathematics, Yale University, 219 Prospect St, New Haven, CT 06511, USA \\ \texttt{siqing.zhang@yale.edu}}

% Prevents the current date from appearing automatically
\date{} 
\maketitle

\begin{abstract}
In this paper, we show that for any reductive group $G$ the moduli space of semistable $G$-Higgs bundles on a curve in characteristic $p$ is a twisted form of the moduli space of semistable flat $G$-connections. 
As a consequence, we show that the Decomposition Theorem for the Hitchin morphism for $G$-Higgs bundles has the same shape as that for the de Rham-Hitchin morphism for flat $G$-connections. 
\end{abstract}


 \section{Introduction}
When it comes to semistability, there is a discrepancy between the Non Abelian Hodge Theorems (NAHTs) in characteristics 0 and $p$ established so far.
Indeed, fix a reductive group $G$ and a smooth projective connected curve $C$ with genus $g(C)\geq 2$ over an algebraic closed field $k$. 
When $k=\mathbb{C}$, the NAHT, proven by Simpson in \cite{simpson-repnII}, establishes a diffeomorphism between the moduli space $\schHig$ of semistable Higgs $G$-bundles on $C$ and the moduli space of semistable flat $G$-connections $\schdR.$
When $char(k)=p>0,$ the NAHT established by Chen-Zhu in \cite{chen-zhu} states that the moduli stack $\stdR(C)$ of flat $G$-connections on $C$ is a twisted version $\stsp\times^{\stP}\stHig(C')$ of the moduli stack $\stHig(C')$ of Higgs $G$-bundles on the Frobenius twist $C'$, where $\stP$ is the Picard stack of symmetries of the Hitchin fibration and $\stsp$ is a torsor under $\stP.$

The Chen-Zhu isomorphism $\stsp\times^{\stP}\stHig(C')\xrightarrow{\sim} \stdR(C)$ does not preserve semistability. Indeed, one can already see this in the case of $G=GL_2$, where an object in $\stsp$ corresponds to a splitting of the Bezrukavnikov-Mirkovic-Rumynin Azumaya algebra \cite{BMR08} on a degree 2 spectral curve over $C'$.
If the splitting is not very good in the sense of \cite[\S3.3]{dCGZ}, then the isomorphism between the de Rham-Hitchin and Dolbeault-Hitchin fibers does not preserve semistability, see \cite[Remark 4.10, last paragraph on p.1376]{dCGZ}, where it is shown that a non-very-good splitting can send a degree-$d$ Higgs bundle to a degree $pd$ connection while sends a degree-$d'$ Higgs sub-bundle to a large degree sub-connection so that the resulting connection is not semistable.

The first result in this paper is the restoration of semistability when we replace the Picard stack $\stP$ by its neutral component $\stP^o$, and replace the torsor $\stsp$ by a corresponding smaller piece $\stsp^o$, which is a $\stsp^o$-torsor.

\begin{introthm}[=\Cref{thm: ss main}] \label{thm: main thm intro} 
Fix a degree $d \in \pi_1(G)$. Chen-Zhu's NAHT isomorphism restricts to an isomorphism between semistable moduli stacks
\[\stsp^o\times^{\stP^o}\stHig^{ss}(C',d)\xrightarrow{\sim}\stdR^{ss}(C,pd).\]
When $p> 2h(G)-2$, where $h(G)$ is the Coxeter number of $G$, every term above admits a quasi-projective adequate moduli space in the sense of Alper \cite{alper_adequate}, and we obtain a schematic version of the NAHT isomorphism.
\end{introthm}

As a consequence of \Cref{thm: main thm intro}, we can relate the Decomposition Theorems obtained by pushing forward the intersection complex $\IC$ under the Hitchin morphism $h_{Dol}: \schHig(C', d)\to A(G, \omega_{C'})$ and the de Rham-Hitchin morphism $h_{dR}: \schdR(C,dp)\to A(G, \omega_{C'})$. 

  \begin{introthm}[= \Cref{thm: dt}] \label{thm: intro thm dt} Suppose that $p>2h(G)-2$. Fix a degree $d \in \pi_1(G)$ and a prime $\ell$ distinct from $p$. Then,
     \begin{enumerate}
            \item There is a distinguished isomorphism $h_{Dol,*}\IC\cong h_{dR,*}\IC$ in $D^b_c(A(G, \omega_{C'}),\oql)$.
         \item There are canonical isomorphisms of cohomology and perverse cohomology sheaves:
         \[\mathcal{H}^i(h_{Dol,*}\oql)\cong \mathcal{H}^i(h_{dR,*}\oql), \; \; \pcs^*(h_{Dol,*}\oql)\cong \pcs^*(h_{dR,*}\oql).\]
     \end{enumerate}
 \end{introthm}

The proofs of the above results are very different from the $GL_n$-case as in \cite{dCGZ}, which relies on explicit Riemann-Roch calculations. 
Such calculations are impossible for principal bundles. 
Worse still, in contrast to the construction of moduli spaces as in the authors' previous paper \cite{hererro-zhang}, where we essentially reduce to the $GL_n$-case, 
no such reduction seems viable for the proof of the semistable NAHT.
Some new machinery, e.g. the theory of $\Theta$-semistability as in \cite{hl-instability, heinloth-hilbertmumford}, proves to be indispensible in our approach in this paper.
Furthermore, the proof that $\stsp^o$ is a $\stP^o$-torsor involves new results on the torsion components of $\stsp$ which are not present in $GL_n$-case. 
Finally, to show the quasi-projectivity of the good moduli space of $\stP^o,$ we give an algebraic proof of the stability of the Kostant section, which seems to be new in the literature.

 \textbf{Acknowledgements.}
 We thank Mark de Cataldo, Roberto Friguelli, Mirko Mauri, Sasha Petrov, Xiao Wang, and Daxin Xu, for useful discussions and comments. This material is based upon work supported by the National Science Foundation under Grant No. DMS-1926686.

 \subsection{Notation}\label{1stnot}
 Let $k$ be an algebraically closed field of characteristic $p>0.$
Fix a smooth connected projective curve $C$ over $k.$
We denote by $\omega_C$ the canonical line bundle on $C.$
    We assume throughout this paper that the genus $g(C)$ of $C$ satisfies $g(C) \geq 2$.

 Let $G$ be a reductive group over $k.$ We fix the choices of a maximal torus $T \subset G$ and a Borel subgroup $B \supset T$.
 We denote by $\fg,\ft,\fb$ the corresponding Lie algebras, which we often view as affine vector space schemes. We write $Z_G \subset G$ for the center of the group $G$.
 We use $\fc$ to denote the GIT quotient scheme $\fg\git G$ by the adjoint action. We denote by $W$ be the Weyl group of $G$, and we write $\Phi$ for the set of roots of $G$ with respect to $T$. 

     We assume throughout this paper that $p\nmid |W|.$
 Let $h(G)$ be the Coxeter number of $G$ as defined in \cite[\S5.1]{serre2005complete}.
 In some places we also need the assumption that $p\geq h(G)$, which is stronger than $p\nmid |W|$ except in cases $A_{p-1}$ and $D_{2=p}$.
In the places where we need the existence of certain moduli spaces constructed in \cite{hererro-zhang}, we often require $p>2h(G)-2$, which is called as property Low Height in \cite{hererro-zhang}.
 
We use $F$ to denote the absolute Frobenius morphism on $\bF_p$-schemes. Given any $k$-scheme $X,$ we set $X':=X\times_{k,F}k$ to be the Frobenius twist of $X.$ We write $Fr_{X/k}:X\to X'$ for the relative Frobenius over $k$, and we omit parts of the subscript when it is clear from the context. 
Given any other $k$-scheme $S,$ we often denote $X_S:=X\times_k S$.

In this paper, we often use twisted products of stacks under the action of a Picard stack.
See \cite[\S A.1]{chzh17} for the notion of Picard stacks.
Let $\mathcal{P}$ be a Picard stack over a site $\mathcal{S}$.
Let $\mathcal{H}$ be a stack over $\mathcal{S}$ with a bifunctor $a:\mathcal{P}\times\mathcal{H}\to \mathcal{H}$.
We say that $a$ defines an action of $\mathcal{P}$ on $\mathcal{H}$ if the condition (i) in \cite[\S A.5]{chzh17}.
We say that $a$ makes $\mathcal{H}$ a $\mathcal{P}$-torsor if the conditions (i)-(iii) in \cite[\S A.5]{chzh17} are satisfied.
Given two stacks $\mathcal{H}$ and $\mathcal{M}$ with a $\mathcal{P}$-action, we denote by 
$   \mathcal{H}\times^{\mathcal{P}}\mathcal{M}$ the 2-categorical quotient of the anti-diagonal action of $\mathcal{P}$ on $\mathcal{H}\times\mathcal{M}$ as in \cite[p. 419, Compl\'ement]{ngo2006fibration}.
By \cite[Lemma 4.7]{ngo2006fibration}, we have that, when $\mathcal{H}$ is a $\mathcal{P}$-torsor, $\mathcal{H}\times^{\mathcal{P}}\mathcal{M}$ is equivalent to a 1-stack, which we call the twisted product of $\mathcal{H}$ and $\mathcal{M}$.

\section{Very good \texorpdfstring{$G$}{G}-splittings}
 In \S\ref{subsec: review chzh}, we review Chen and Zhu's results in \cite{chen-zhu}.
 In \S\ref{sec: conn}, we show that the torsion primes of $\pi_0(\stP)$ divide $|\pi_0(Z_G)|$ and $|W|.$
  In \S\ref{sec: vgsp}, we define the open substack $\stsp^o \subset \stsp$ and study the sheaf of connected components of $\stsp$. In \S\ref{subsection: torsor} we show that $\stsp^o$ is a torsor under $\stP^o.$

 \subsection{Review of Chen-Zhu's Non Abelian Hodge Isomorphism}\label{subsec: review chzh}

\begin{notn}
    The natural $\bG_m$-action on $\fg$ descends to $[\fg/G]$ and $\fc$,  so that we can twist them by $\bG_m$-torsors over $C.$
For any line bundle $L$ on $C$, we denote by $L^{\times}:=\mathrm{Spec}_{\cO_C}(\bigoplus_{n\in\bZ}L^{\otimes n})$ the associated $\bG_m$-torsor, and we denote the corresponding twists $(-)\times^{\bG_m}L^{\times}$
by $[\fg/G]_L, \;\;\fc_L, \;\;\fg_L$, etc., which naturally live over the curve $C$.
\end{notn}

\noindent \underline{\textit{Dolbeault moduli space.}} Let $L$ be a line bundle on $C$. The Dolbeault stack $\stHig(C,G,L)$ is the stack of $C$-sections of $[\fg/G]_L \to C$.
The Hitchin base $\hbas(C,G,L)$ is the affine space parametrizing $C$-sections of $\fc_L \to C$, see \cite[\S 5.2]{dalakov2017lectures}. 
The Hitchin morphism $h:\stHig(C,G,L)\to \hbas(C,G,L)$ is the morphism of stacks induced by the natural $\bG_m$-equivariant morphism $\chi: [\fg/G]\to \fc.$
 We also write $\chi:\fg\to \fc$ for the good quotient morphism. We may drop the decorations such as $G$ and $L$ when they are clear from the context.

\noindent \underline{\textit{Centralizers.}} Let $\kappa:\fc\to [\fg/G]$ be the Kostant section.
Let $I\to [\fg/G]$ be the inertia group stack.
We also use the same notation $I \to \mathfrak{g}$ to denote its pullback via $\fg\to [\fg/G]$.
Let $J:=\kappa^*I$ be the regular centralizer.
It is a smooth commutative group scheme over $\fc$ \cite[Lem. 2.1.1]{ngo-lemme-fondamental}.
There is a natural morphism $a: \chi^*J\to I$ of group schemes on $[\fg/G]$ as in \cite[Lem. 2.1.1]{ngo-lemme-fondamental}. 
Since $J$ and $I$ are $\bG_m$-equivariant over $\fc$ and $[\fg/G]$ respectively, we can twist them by the $\bG_m$-torsor $L^{\times}$ and obtain group schemes $J_{L}$ and $I_{L}$ over $\fc_L$ and $[\fg/G]_L$ respectively. 
 As explained in \cite[\S2.2.3]{ngo-lemme-fondamental}, given any line bundle $L$ on $C$ and a choice of square root $L^{1/2}$ of $L,$ we have a corresponding Kostant section $\eta_{\kappa}: A(G,L)\to \stHig(C,G,L)$ of the Hitchin morphism.

\noindent \underline{\textit{The Picard stack.}} 
    We define $\stP(L)\to A(G,L)$ to be the relative stack of $J_L$-torsors on the trivial family $C_{A(G,L)} \to A(G,L)$.
We refer to $\stP(L)$ as the Picard stack. It is an algebraic stack, and the structure morphism $\stP(L) \to A(G,L)$ is smooth by \cite[Prop. 4.3.5]{ngo-lemme-fondamental}.

\noindent \underline{\textit{Action of the Picard stack.}} 
Let $S$ be a $k$-scheme. The data of an $S$-point of $\stHig(C,G,L)$ corresponds to a section $(E,\phi): C_S\to ([\fg/G]_L)_S.$
The pullback group scheme $(E,\phi)^*I_{L}$ is isomorphic to the relatively affine group scheme of Higgs bundle automorphisms $Aut(E,\phi) \to C_S$. 
Let us denote by $b=h(E,\phi)\in A(G,L)(S)$ the image under the Hitchin morphism, 
which corresponds to a section $b: C_S\to (\fc_L)_S.$
Set $J_b:= b^*J_L.$
Then the morphism $a:\chi^*J\to I$ induces a homomorphism
\begin{equation}
    \label{eqn: a}
    a_{(E,\phi)}: J_b\to Aut(E,\phi)
\end{equation}
 of relatively affine group schemes on $C_S$.
 Given a $J_b$-torsor $F$, we can then form another Higgs bundles $F\times^{J_b, a_{(E,\phi)}}(E,\phi)$, which also lies over $b$.
 This defines the action of $\stP(L)$ on $\stHig(C,G,L)$.

\noindent \underline{\textit{De Rham stack.}} The de Rham stack $\stdR(C,G)$ is defined to be the stack of flat $G$-connections, i.e., $G$-torsors equipped with flat connections.
Given a flat $G$-connection $(E,\nabla)$ on $C,$ the $p$-curvature $\Psi(\nabla)$ is a $\omega_{C}^{\otimes p}$-twisted Higgs field on $E,$ see \cite[\S A.6]{chen-zhu}.
By \cite[Prop. 3.1]{chen-zhu}, there is a de Rham-Hitchin morphism $h_{dR}: \stdR(C)\to A(\omega_{C'})$ fitting in the following commutative diagram of $k$-stacks:
\begin{equation*}
    \xymatrix{
    \stdR(C,G) \ar[r]^-{\Psi} \ar[d]_-{h_{dR}} &
    \stHig(C,G,\omega_{C}^{\otimes p}) \ar[d]^-{h}\\
    A(G,\omega_{C'}) \ar[r]_-{Fr^*} & A(G,\omega_{C}^{\otimes p}),
    }
\end{equation*}
where $\Psi$ is given by taking $p$-curvatures, and
$Fr^*$ is given by Frobenius pullback of sections.

 In this subsection we review the constructions in \cite{chen-zhu}, leading to the stack $\stsp(G)$ of what we call $G$-splittings, and the Non Abelian Hodge isomorphism.
 
 \noindent \underline{\textit{The tautological section $\tau$.}}
Consider the scheme of Lie algebras $Lie(I) = \Spec_{\fg}(\Sym^{\bullet}(\Omega^1_{I/\fg}))$ over $\fg$, which may be alternatively described as
\[ Lie(I) = \{ \, x, y \in \mathfrak{g} \, \mid \, [x,y]=0 \, \} \subset \mathfrak{g}^2 \xrightarrow{pr_2} \mathfrak{g}.\]
There is a tautological section $\tau_0:\fg\to Lie(I)$ induced by $x\mapsto x.$
 By \cite[Lem. 2.2]{chen-zhu}, this section $\tau_0$ induces a tautological section $\tau:\fc\to Lie(J).$

\noindent \underline{\textit{The $J$-Hitchin System.}}
The section $\tau:\fc\to Lie(J)$ is $\bG_m$-equivariant for the natural $\bG_m$-actions induced by the diagonal action on $Lie(I)\subset \fg\times \fg.$
Therefore, given any line bundle $L$ on $C,$ we can twist $\tau$ by the $\bG_m$-torsor $L^{\times}$ to obtain $\tau(C,L): \fc_L \to Lie(J)_L$.

\begin{defn}
    The $J$-Hitchin base $A(J,L)$ is the $A(G,L)$-functor that sends an $ A(G,L)$-scheme $b:S \to A(G,L)$, corresponding to a section $b: C_S\to (\fc_L)_S$, to the set $A(J,L)(b)$ of $C_S$-sections of $b^*\Big(Lie(J)_L\Big)\cong Lie(J_b)_L$ \cite[p.1713]{chen-zhu}.
\end{defn}

\begin{defn}\label{jhit}
The $J$-Hitchin system is the following diagram of $A(G,L)$-stacks:
\begin{equation}\label{JHit}
\xymatrix{
    \stHig(C,J,L) \ar[rr]^-{h(C,J,L)}&& A(J,L) \ar@/^/[rr]^-{p(C,L)} && A(G,L) \ar@/^/[ll]^-{\tau(C,L)},
} 
\end{equation}
where, given any $k$-scheme $S$ and any $b\in A(G,L)(S)$ (giving rise to $b: C_S\to \fc_L$):
\begin{enumerate}
    \item $\stHig(C,J,L)(b)$ is the groupoid of $J_b$-Higgs bundles $Sect(C_S,[Lie(J_b)/J_b]_L);$
    \item $h(b)$ is induced by the natural morphism $[Lie(J_b)/J_b]_L\to (Lie(J_b)\git J_b)_L=Lie(J_b)_L;$
    \item $p$ is the structure morphism of the $A(G,L)$-functor $A(J,L)$;
    \item $\tau(b)$ is the pullback $b^*$ of the twisted section 
    $\tau_L: \fc_L\to Lie(J)_L.$
\end{enumerate}
\end{defn}
\begin{remark}
    The pullback $\cM_{Dol}(C,J,L)\times_{A(J,L),\tau(C,L)} A(G,L)$ is naturally identified with the Picard stack $\stP(L)$.
    Indeed, let $b:S\to A(G,L)$ and put $J_b=b^*J_L$. Since $J_b$ is commutative, its adjoint action on $Lie(J_b)$ is trivial, so an object of $\stHig(C,J,L)(b)$ is a pair $(F,\varphi)$, where $F$ is a $J_b$-torsor on $C_S$ and $\varphi\in \Gamma(C_S,Lie(J_b)_L)$. The map $h(C,J,L)$ sends $(F,\varphi)$ to $\varphi$. Base changing by $\tau(C,L)$ imposes $\varphi=b^*\tau_L$, and leaves exactly the $J_b$-torsor $F$. This is functorial in $S$ and $b$, hence gives the desired identification with $\stP(L)$.
\end{remark}

\noindent \underline{\textit{The de Rham-$J$-Hitchin System.}}
\begin{notn}
    For any $k$-scheme $S$ and any $b'\in A(G,\omega_{C'})(S)$, 
we set $b^p:=Fr_{C_S/S}^*b'\in A(G,\omega_{C}^{\otimes p})(S).$ We set $J^p$ to be the smooth commutative group scheme over $C\times A(G, \omega_{C'})$ given by the fiber product $J^p := J_{\omega_{C}^{\otimes p}} \times_{A(G,\omega_{C}^{\otimes p}), Fr^*} A(G, \omega_{C'})$.
\end{notn}

The affine group scheme $J_{b^p}:=(b^p)^*J_{\omega_{C_S/S}^{\otimes p}}\cong J_{b'} \times_{C_S'} C_S$ on $C_S$ admits the Cartier connection $\nabla^{can}$ \cite[Thm. 5.1]{katz1970}.
Therefore, the notion of $J_{b^p}$-connections (i.e., $J_{b^p}$-torsors with equipped with connections) makes sense, see the formalism in \cite[\S A]{chen-zhu}.
\begin{defn}
    We define $\stdR(C,J^p)$ to be the $A(G,\omega_{C'})$-stack whose fiber over $b'\in A(G,\omega_{C'})(S)$ is the groupoid of 
flat $J_{b^p}$-connections $(E,\nabla)$ on $C_S/S.$
\end{defn}
By taking the $p$-curvature of such a $(E,\nabla)$ as above, we obtain a $Fr^*\omega_{C'_S/S}$-twisted $J_{b^p}$-Higgs bundle $(E,\Psi(\nabla))$ on $C_S.$
The Higgs field $\Psi(\nabla)$ defines a section $C_S\to Lie(J_{b^p})_{Fr^*\omega_{C'_S/S}}\cong Fr^*\Big( Lie(J_{b'})_{\omega_{C'_S/S}}\Big).$
By \cite[\S A.8]{chen-zhu} There is a section $\Psi'(\nabla):C'_S\to Lie(J_{b'})_{\omega_{C'_S/S}}$ such that we have an identity of sections $\Psi(\nabla)=Fr^*\Psi'(\nabla).$

\begin{defn}\label{defn: j-dr}
    The $J$-de Rham-Hitchin system is the following diagram of $A(G,\omega_{C'})$-stacks:
    \begin{equation}
        \xymatrix{
        \stsp(G) \ar[rr] \ar[d] && \stdR(C,J^p) \ar[r]^-{\Psi} \ar[d]^-{h_{dR}^J} & \stHig(C,J, \omega_{C}^{\otimes p}) \ar[d]^-{h(C,J,\omega_{C}^{\otimes p})} \\
        A(G,\omega_{C'}) \ar@/^/[rr]^-{\tau(C',\omega_{C'})} &&
        A(J,\omega_{C'}) \ar[r]^-{Fr^*} \ar@/^/[ll]^-{p(C',\omega_{C'})} &
        A( J, \omega_{C}^{\otimes p}),
        }
    \end{equation}
    where
    \begin{enumerate}
        \item $h(C,H,\omega_{C}^{\otimes p}),\tau(C',\omega_{C'}),p(C',\omega_{C'})$ are defined in Definition \ref{jhit};
        \item $\Psi$ takes a $J_{b^p}$-flat connection $(E,\nabla)$ to the Higgs bundle $(E,\Psi(\nabla))$;
        \item the existence of the morphism $h_{dR}^J$ follows from \cite[\S A.8]{chen-zhu};
        \item the $A(G,\omega_{C'})$-stack $\stsp(G)$ is defined so that the inner left square, with bottom horizontal arrow as $\tau(C',\omega_{C'})$, is Cartesian. 
    We call $\stsp(G)$ the stack of $G$-splittings.
    \end{enumerate}
\end{defn}

\begin{remark}
\label{curlyd}
    We name $\stsp(G)$ as the stack of $G$-splittings because of the following:
    In the $G=GL_n$-case, the crystalline differential operators $D_C$ on $C$ gives rise an Azumaya algebra $\cD$ on the cotangent bundle $T^*C'.$ 
    By \cite[Rmk. 3.13]{chen-zhu}, the stack $\stsp(GL_n)$ coincides with the stack of the splittings of the restrictions of $\cD$ to the spectral curves, as studied in \cite[\S3.4]{groechenig-moduli-flat-connections} and \cite[\S2.2]{dCGZ}.
    Moreover, as mentioned in \cite[\S3.4]{chen-zhu}, for general reductive $G,$ the stack $\stsp(G)$ also coincides with the splitting of the gerbe $\mathcal{G}_{\tau'}$ defined in \cite[Prop. A.7]{chen-zhu}.
\end{remark}

\begin{lemma}[Properties of the stack $\stsp(G)$ of $G$-splittings]
\label{cztech} \;

    \begin{enumerate}
        \item $\stsp(G)$ is smooth and surjective over $A(G,\omega_{C'});$
        \item $\stsp(G)$ is a torsor under the Picard stack $\stP(\omega_{C'})$. \qed
    \end{enumerate}
\end{lemma}

The smoothness is proved in \cite[Lem. 3.7]{chen-zhu}, and the surjectivity is proved in \cite[\S3.4]{chen-zhu}.
The fact that it is a torsor is \cite[Thm. 3.8]{chen-zhu}.
The action of $\stP(\omega_{C'})$ on $\stsp(G)$ is defined as follows. Choose a point $b'$ of $A(G,\omega_{C'})$, let $F'$ be a $J_{b'}$-torsor on $C'$, and let $(E,\nabla)$ be a $J_{b^p}$-flat connection on $C.$
Then the action is defined as $F\cdot (E,\nabla):= (Fr^*F\times^{J_{b^p}} E,\nabla^{can}\otimes \nabla),$
where $\nabla^{can}$ is the Cartier descent connection, see \cite[\S A.0.6]{de2025logarithmic} for the definition of the connection on the twisted product.

 \begin{thm}\cite[Thm. 3.12]{chen-zhu}
     There is a $\stP(\omega_{C'})$-equivariant morphism of $A(G,\omega_{C'})$-stacks
     \[
\widetilde{\fC}:\stsp(G)\times_{A(G,\omega_{C'})} \stHig(C',G,\omega_{C'})\to \stdR(C,G),
    \]
    inducing an isomorphism of $A(G,\omega_{C'})$-stacks
    \begin{equation}\label{cziso}
    \fC: \stsp(G)\times^{\stP(\omega_{C'})} \stHig(C',G,\omega_{C'})\xrightarrow{\sim} \stdR(C,G),
\end{equation}
where the source is the twisted product as in Subsection \ref{1stnot}.
 \end{thm}
Let us recall the construction of $\widetilde{\fC}.$
 Let $b'\in A(G,\omega_{C'})(S),$ let $(E,\nabla)\in \stsp(G)(b'),$ and let $(F,\phi)\in \stHig(C',G,\omega_{C'})(b').$ Then $\widetilde{\fC}((E,\nabla),(F,\phi)):=(E\times^{J_{b^p}}_{a_{Fr^*(F,\phi)}} Fr^*F, \nabla\otimes \nabla^{can}),$ where the subscript $a$ is defined as in equation \eqref{eqn: a}. 

\subsection{Connected components}\label{sec: conn}\;

In this subsection, we prove some preparatory lemmas on the group of connected components of the Picard stack of symmetries of Hitchin fibration $\stP(\omega_{C'}).$
In particular, we show that, under the assumption that $p\nmid |W|,$ the abelian group $\pi_0(\stP(\omega_{C'}))$ is $p$-torsion free.

By \cite[Prop. 2.2.1]{romagny2011composantes}, the Picard stack $\stP(L)$ admits a relative identity component $\stP^o(L)$, given by the neutral connected components over geometric points of $A(G,L)$.
In the case when $G=GL_n,$ the stack $\stP^o(L)$ is given by line bundles of multi-degree $0$ on the spectral curves.
By the openness of $\stP^o(L)$ in $\stP(L)$ and the smoothness of $\stP(L)$ \cite[Prop. 4.3.5]{ngo-lemme-fondamental}, we have that the stack $\stP^o(L)$ is smooth over the Hitchin base $A(G,L).$

To state Definition \ref{defn: pi_0 P} below, we need to justify taking quotients by a Picard stack. 
The quotient of a groupoid $X$ by a Picard category $Q$ is discussed in \cite[Lem. 4.7]{ngo2006fibration}, which states that if the morphism $Aut(1_Q)\to Aut(x)$ is injective for every object $x$ of $X,$ then the 2-categorical quotient $X/Q$ is equivalent to a 1-category. There is also a criterion for when the quotient is equivalent to a set.
\begin{lemma}
\label{lemma: quot set}
    Let $Q$ be a Picard category acting on a groupoid $X.$ 
    If, for any objects $q$ of $Q$ and $x$ of $X,$ the morphism of sets $act(-,1_x): Hom_Q(1_Q, q)\to Hom_X(x, qx)$ is bijective, then the quotient 2-category $X/Q$ is equivalent to a set.
\end{lemma}
\begin{proof}
    By \cite[Lem. 4.7]{ngo2006fibration} and the discussion before it, it suffices to show that given any two objects $x_1, x_2$ in $X,$ two objects $q_1, q_2$ in $Q,$ and two morphisms $\alpha_1: q_1x_1\to x_2$ and $\alpha_2: q_2x_1\to x_2$ in $X,$ there exists a morphism $\beta: q_1\to q_2$ such that the triangle formed by $act(\beta, 1_{x_1}): q_1x_1\to q_2x_1, \alpha_1,$ and $\alpha_2$ is commutative in $X$.
In other words, we hope to
find a $\beta: q_1 \to q_2$ such that $act(\beta,1_{x_1}) = \alpha_2^{-1} \circ \alpha_1$.
     We are reduced to show that the surjectivity of $act(-, 1_{x_1}): Hom_Q(q_1, q_2)\to Hom_X(q_1x_1, q_2x_1).$ 
    The action of $1_{q_1^{-1}}$ induces bijections $Hom_Q(q_1,q_2)= Hom(1_Q,q_1^{-1}q_2)$ and $Hom_X(q_1x_1,q_2x_1)=Hom_X(x_1, q_1^{-1}q_2x_1).$ The desired surjectivity then follows from the assumption.
\end{proof}

\begin{lemma}
\label{lemma: quot sub pic}
    Let $Q$ be a full Picard subcategory of a Picard category $X.$
    The quotient 2-category $X/Q$ is equivalent to a set $\wt{X/Q}$.
    Furthermore, the Picard category structure on $X$ induces an abelian group structure on $\wt{X/Q}.$
\end{lemma}
\begin{proof}
    Given any object $x$ of $X$ and object $q$ of $Q,$ the morphism $act(-,1_{x}): Hom_Q(1_Q,q)=Hom_X(1_X,q)\to Hom_X(x,qx)$ is a bijection with inverse given by $act(-, 1_{x^{-1}}).$ The first statement then follows from \Cref{lemma: quot set}. The second statement can be checked directly.
\end{proof}

\begin{defn}
\label{defn: pi_0 P}
We denote by $\pi_0(\stP(L))$ the sheaf of abelian groups on the big \'etale site of $A(G,L)$ associated to the quotient functor $\stP(L)/\stP^o(L)$ as in \Cref{lemma: quot sub pic}.
\end{defn}

For any geometric point $b$ of $A(G,L)$, the restriction $\pi_0(\stP(L))|_{b}$ is the constant sheaf associated to the group of components $\pi_0(\stP(L)_b)$ of the smooth group stack $\stP(L)_b$ over $b$.

\begin{lemma} \label{lemma: geometric points vs sheaf}
    Let $U$ be a scheme over $A(G,L)$ and let $s,s' \in \pi_0(\stP(L))(U)$. Then $s = s'$ if and only if for all geometric points $b$ of $U$, we have equality of pullbacks $s|_{b} = s'|_{b}$.
\end{lemma}
\begin{proof}
    By subtracting $s-s'$, we may assume that $s'=0$. After passing to an \'etale cover of $U$, we may assume that $s$ comes from a section $\widetilde{s}: U \to \stP(L)_{U}$. Our assumption implies that for all geometric points $b$ of $U$, the restriction $\widetilde{s}|_{b}: b \to \stP(L)_b$ lands in the open substack $\stP^o(L)_b$. Therefore, the section $\widetilde{s}$ lands in $\stP^o(L)_U$, and it follows that $s=0$.
\end{proof}

\begin{prop}\label{prop: p torsion prop}
Under our assumption that $p\nmid |W|,$ we have that for any geometric point $b$  of $A(G,L),$ the group $\pi_0(\stP(L)_b)$ is $p$-torsion free.
\end{prop}
\begin{proof}
Let $J_b^o$ be the neutral component of $J_b.$ Consider the exact sequence of smooth commutative $C_b$-group schemes
\begin{equation} \label{eqn: short exact sequence}
    1\to J_b^o\to J_b\to \pi_0(J_b)\to 1
\end{equation}
where $\pi_0(J_b) \to C_b$ is \'etale.
By \cite[Cor. 2.3.2]{ngo-lemme-fondamental}, there is a surjection $\pi_0(Z_G) \times C_b \to \pi_0(J_b).$
Since $\pi_0(Z_G)$ is finite and is automatically $p$-torsion free, we have that $\pi_0(J_b)$ is $p$-torsion free. 

The short exact sequence \eqref{eqn: short exact sequence} induces morphisms of group stacks of torsors 
\begin{equation} \label{eqn: short exact sequence torsor stacks}
    \Bun_{J_b^o}(C_b) \xrightarrow{f} \Bun_{J_b}(C_b) = \stP(L)_b  \xrightarrow{g} \Bun_{\pi_0(J_b)}(C_b).
\end{equation}

\noindent \textbf{Claim.} Fix a choice of $\pi_0(J_b)$-torsor $F$ on $C_b$ corresponding to a morphism of functors $b \to \Bun_{\pi_0(J_b)}(C_b)$. Then subset $g^{-1}(F) \subset |\stP^o(C,L)_b|$ given by the image of the fiber product $\stP^o(C,L)_b \times_{\Bun_{\pi_0(J_b)}(C_b)} b \to \stP^o(C,L)_b$ is open at the level of geometric points.

Let us briefly explain the proof of the \textbf{Claim}, which is standard, but there is some care needed because $\pi_0(J_b) \to C_b$ is not separated a priori. The fiber product $\stP^o(C,L)_b \times_{\Bun_{\pi_0(J_b)}} b$ is the functor whose $S$-points consist of pairs $(E, \psi)$ of a $J_b$-torsor $E$ on $C_S$ lying in $\stP^o(C,L)$ and a section $\psi: C_S \to \underline{\text{Iso}}(E, F|_{C_S})$ of the relatively finite type \'etale scheme $\underline{\text{Iso}}(E, F|_{C_S}) \to C_S$ classifying isomorphisms from $E$ to $F|_{C_S}$. A standard argument spreading out the sections $\psi$ shows that $g^{-1}(F)$ is locally constructible. To conclude openness, we need to show that it is closed under generalization. Since the stack $\stP^0(C,L)_b$ is locally Noetherian, this is equivalent to the following: for all complete discrete valuation rings $R$ and morphisms $\Spec(R) \to \stP^0(C,L)_b$, if the image of the special point lies on $g^{-1}(F)$, then the same holds for the image of the generic point. To see this, let $E$ denote the corresponding torsor on $C_{\Spec(R)}$, and let $s$ denote the special point of $\Spec(R)$. By assumption, after perhaps extending the residue field, there is a section $\psi: C_s \to \underline{\text{Iso}}(E, F|_{C_{\Spec(R)}})|_{C_s}$. Since $\underline{\text{Iso}}(E, F|_{C_{\Spec(R)}}) \to C_{\Spec(R)}$ is \'etale, we may extend this section over all nilpotent thickenings of $s$ in $\Spec(R)$. Then we may use Grothendieck's existence theorem, which holds for targets that are quasi-separated thanks to a Tannakian argument \cite[\href{https://stacks.math.columbia.edu/tag/0GHK}{Tag 0GHK}]{stacks-project}, to define a section $\widetilde{\psi}: C_{\Spec(R)} \to \underline{\text{Iso}}(E, F|_{C_{\Spec(R)}})$, which shows that the image of $\Spec(R)$ is contained in $g^{-1}(F)$. The \textbf{Claim} follows.

The \textbf{Claim} implies that any geometric point in the neutral component $\stP^o(C,L)_b$ lies in the image of $f$. Indeed, if this was not the case, then by the exactness on the left of \eqref{eqn: short exact sequence torsor stacks} at the level of geometric points, this would mean that the open $g^{-1}(\text{triv})$ for the trivial $\pi_0(J_b)$-torsor does not equal $\stP^o(C,L)_b$. Hence, there should be another $\pi_0(J_b)$-torsor $F$ with $g^{-1}(F) \neq \emptyset$. The two opens $g^{-1}(F)$ and $g^{-1}(\text{triv})$ would then be disjoint and nonempty in the connected and smooth (hence integral) stack $\stP^o(C,L)_b$, a contradiction.

It then follows that the $p$-torsion of $\pi_0(\stP(L)_b)$ comes from torsors in the image of the morphism $f$ in \eqref{eqn: short exact sequence torsor stacks}. Indeed, if the image of a $J_b$-torsor $E$ is $p$-torsion in $\pi_0(\stP(L))$, it means that its $p^{th}$ power $E^p$ is in $g^{-1}(\text{triv})$ (here we are using that $\stP^o(C,L)_b$ is contained in $g^{-1}(\text{triv})$). In other words, the associated $\pi_0(J_b)$-torsors $E^p(\pi_0(J_b))$ is trivializable. But, since $\pi_0(J_b)$ is $p$-torsion-free, this implies that also $E(\pi_0(J_b))$ is trivializable. Hence, $E$ is in $g^{-1}(\text{triv})$, which is exactly the image of $f$.

We conclude that the $p$-torsion in $\pi_0(\stP(L)_b)$ lies in the image of the induced map on component groups $\pi_0(f): \pi_0(\Bun_{J_b^o}(C_b)) \to \pi_0(\stP(L)_b)$. We are reduced to showing that $\pi_0(\Bun_{J_b^o}(C_b))$ is $p$-torsion free, which is proven in \Cref{lemma: torsions} below.
\end{proof}

\begin{lemma}
\label{lemma: torsions}
    The torsion primes in $\pi_0(\Bun_{J_b^o}(C_b))$ divide $|W|.$
\end{lemma}
\begin{proof}
    Let $\pi:\widetilde{C}\to C_b$ be the cameral curve associated to $b.$
    By \cite[Prop. 2.4.7]{ngo-lemme-fondamental}, we have a natural identification $J_b^o=\Big(\pi_*(T\times\widetilde{C})^W\Big)^o,$ the neutral component of the $W$-invariant part of the Weil restriction of the group scheme $T\times\widetilde{C}$ over $\widetilde{C}.$
    Let $Nm: \pi_*(T\times\widetilde{C})\to \pi_*(T\times\widetilde{C})^W$
    be the norm morphism sending a section $s$ to $Nm(s):=\prod_{w\in W} w(s).$
    Since $\wt{C}$ is finite flat over $C_b,$ \cite[Prop. A.5.11.(3)]{conrad2015pseudo} entails that $\pi_*(T\times\widetilde{C})$ has geometrically connected fibers over $C_b.$ Therefore, the morphism $Nm$ factors through the neutral component $J_b^o\hookrightarrow \pi_*(T\times\widetilde{C})^W.$
    The composition $J_b^o\hookrightarrow \pi_*(T\times\widetilde{C})\xrightarrow{Nm} J_b^o$
    coincides with the $|W|$-th power map on $J_b^o.$
    By taking stacks of torsors and connected components, we obtain a factorization
    \[|W|: \pi_0(\Bun_{J_b^o}(C_b)) \to \pi_0(\Bun_{\pi_*(T\times\widetilde{C})}(C_b))\to \pi_0(\Bun_{J_b^o}(C_b)).\]
    By \cite[\S 9.2, Cor 14]{blr-neron}, the N\'eron-Severi group of a proper curve over an algebraically closed field is torsion free. Therefore, $\pi_0$ of
     $\Bun_{\pi_*(T\times\widetilde{C})}(C_b) \cong \Bun_{T}(\widetilde{C}) \cong \mathcal{P}ic(\widetilde{C})\otimes_{\mathbb{Z}} X_*(T)$ is torsion free. Hence, the factorization above implies that $|W|$ kills the torsion of $\pi_0(\Bun_{J_b^o}(C)),$ as desired.
\end{proof}

\begin{prop}
    \label{lemma: the norm}
    For any geometric point $b'$ of $A( G,\omega_{C'}),$ there is a norm morphism $Nm: \stP( \omega_C^{\otimes p})_{b^p} \to \stP( \omega_{C'})_{b'}$ such that the following composition is isomorphic to the $p$-th power map:
    \[\stP( \omega_{C'})_{b'}\xrightarrow{Fr^*} \stP( \omega_C^{\otimes p})_{b^p} \xrightarrow{Nm}\stP( \omega_{C'})_{b'}.\] 
\end{prop}
\begin{proof}
    Without loss of generality, we can change the ground field and assume that $b'$ is a $k$-point. We start by proving (1). Let $\wt{C}' \to C'$ be the cameral curve associated with $b'.$
    Let $\wt{C}^p$ be the base change $\wt{C}'\times_{C'} C,$ which is also the cameral curve over $C$ associated with $b^p.$
    %\siqing{Fix an isomorphism $T\cong \bG_m^{\oplus r}.$}
    Since $C\to C'$ is finite flat, so is $\wt{C}^p\to \wt{C}'.$
    By \cite[\href{https://stacks.math.columbia.edu/tag/0BD2}{Tag 0BD2}, \href{https://stacks.math.columbia.edu/tag/0BCY}{Tag 0BCY}]{stacks-project}, there is a norm morphism on the stack of $T$-bundles
    $Nm: \Bun_T(\wt{C}^p)\to \Bun_T(\wt{C}')$
    such that the composition 
    \[\Bun_T(\wt{C}')\xrightarrow{Fr^*} \Bun_T(\wt{C}^p)\xrightarrow{Nm} \Bun_T(\wt{C}')\]
    is given by taking the $p$-th power. 

    By \cite[Lem. 3.1.3]{chzh17}, for any line bundle $L$ and point $b \in A(G,L)(k)$, the stack $\stP(L)_{b}$ is identified with a stack of  $T$-bundles on the corresponding cameral curve $\wt{C}$ with some extra structures.
    Namely, an object in $\stP(L)_{b}$ corresponds to a tuple of data $(E_T,\{\gamma_w|w\in W\}, \{c_{\alpha}|\alpha\in \Phi\}),$ where 
   
    \noindent $\bullet$ $E_T$ is a $T$ bundle on $\wt{C};$
    
    \noindent $\bullet$ for each $w\in W,$ we have an isomorphism $\gamma_w$ of $T$-bundles:
    \[w(E_T):=(\wt{C}\times_{w^{-1},\wt{C}} E_T)\times^{T,w} T\xrightarrow[\gamma_w]{\sim} E_T.\]
    These $\gamma_w$'s are compatible in the natural way.

    \noindent $\bullet$ for each root $\alpha\in \Phi,$ we have a trivialization $c_{\alpha}$ of the associated $\mathbb{G}_m$-torsor $(E_T|_{\wt{C}_{\alpha}})\times^{T,\alpha}\bG_m,$ where $\wt{C}_{\alpha}$ is the fixed point subscheme of $\wt{C}$ under $s_{\alpha}.$
  
    It suffices to show that the morphism $Nm$ above is compatible with the extra structures. This amounts to showing that for any given object $(E_T,\{\gamma_w|w\in W\},\{c_{\alpha}|\alpha\in \Phi\})$ in $\stP(\omega_{C}^{\otimes p})_{b^p},$ there is a canonical way to endow $Nm(E_T)$ with the data of $\gamma_w$'s and $c_{\alpha}$'s. 
It suffices to show the following three claims: 

    \noindent \textbf{Claim A.}
        There is a canonical isomorphism $w(Nm(E_T))\cong Nm(w(E_T));$
        
    \noindent \textbf{Claim B.} 
        The norm of the trivial $T$-torsor is the trivial $T$-torsor;
    
    \noindent \textbf{Claim C.}
        For every character $\alpha\in X^*(T),$ there is a canonical isomorphism \[Nm(E_T)\times^{T,\alpha}\bG_m\cong Nm(E_T\times^{T,\alpha}\bG_m).\]
    
    Claim B follows from \cite[\href{https://stacks.math.columbia.edu/tag/0BCY}{Tag 0BCY}]{stacks-project}. Claim A in turn follows from the following two subclaims:
    \begin{enumerate}
        \item[A1] There is a canonical isomorphism $Nm(E_T\times^{T,w} T)\cong Nm(E_T)\times^{T,w} T.$ 
        \item[A2] There is a canonical isomorphism $Nm(\wt{C}^p\times_{w,\wt{C}^p} E_T)\cong \wt{C}'\times_{w,\wt{C}'} Nm(E_T).$ 
    \end{enumerate}
(A1): Fix an isomorphism $T\cong \bG_m^{\oplus r}$. Let us show the more general fact that, for any automorphism $x\in GL(r,\bZ)=Aut(T)$ we have $Nm(E_T\times^{T,x} T)\cong Nm(E_T)\times^{T,x} T.$
Using $T\cong \bG_m^{\oplus r},$ we can write $E_T$ as a direct sum of line bundles $\bigoplus_{i=1}^r L_i.$ 
Consider any matrix $A:=(a_{ij})\in GL(r,\bZ).$ We have that 
\[E_T\times^{T,A}T=\bigoplus_{i=1}^r \bigotimes_{j=1}^r L_j^{\otimes a_{ij}}.\]
The desired fact then follows from the multiplicativity of $Nm.$

\noindent (A2): 
For each $w\in W,$ we have a Cartesian diagram of cameral curves:
\[
\xymatrix{
\wt{C}^p\ar[r]^-{w} \ar[d] & \wt{C}^p\ar[d]\\
\wt{C}' \ar[r]_-{w}& \wt{C}'.
}
\]
(A2) then follows from the fact that the Norm map on line bundles is compatible with base change \cite[\href{https://stacks.math.columbia.edu/tag/0BD2}{Tag 0BD2}]{stacks-project}.

The proof of Claim C is similar to the proof of (A1). Namely, for each character $\alpha,$ and $T$-torsor $\bigoplus_{n=1}^r L_n.$ The line bundle $E_T\times^{T,\alpha}\bG_m$ is given by $\bigotimes_{i=1}^r L_i^{\otimes n_i}$ for some $n_i\in\bZ.$
The desired canonical isomorphism again follows from the multiplicativity of $Nm.$
\end{proof}

\begin{coroll} \label{coroll: injectivity frobenius pullback}
    For any geometric point $b'$ of $A( G,\omega_{C'}),$ the morphism induced by Frobenius pullback $Fr^*: \pi_0(\stP( \omega_{C'})_{b'}) \to \pi_0(\stP( \omega_C^{\otimes p})_{b^p})$ is an injection of groups of connected components.
\end{coroll}
\begin{proof}
    By \Cref{lemma: the norm}, the composition $
        \pi_0(\stP( \omega_{C'})_{b'})\xrightarrow{Fr^*} \pi_0(\stP(\omega_C^{\otimes p})_{b^p})\xrightarrow{\pi_0(Nm)} \pi_0(\stP( \omega_{C'})_{b'})$    coincides with multiplication by $p.$ 
    By \Cref{prop: p torsion prop}, we have that $\pi_0(\stP( \omega_{C'})_{b'})$ is $p$-torsion free. 
    Therefore, the composition is injective, and hence the same holds for $Fr^*$.
\end{proof}

 \subsection{Very good \texorpdfstring{$G$}{G}-splittings and components of \texorpdfstring{$\stsp(G)$}{H(C,G)}}\label{sec: vgsp}\;

There is a morphism $forget: \stsp(G)\to \cM_{dR}(C, J^p) \to  \stP(\omega_{C}^{\otimes p})$ given by forgetting the connections.
\begin{defn}[Very Good $G$-Splittings]
    The stack of very good $G$-splittings $\stsp^o(G)$ is the $A(G,\omega_{C'})$-stack defined as the fiber product $\stP^o(\omega_{C}^{\otimes p})\times_{\stP(\omega_C^{\otimes p})} \stsp(G)$.
\end{defn}

When $G=GL_n,$ the stack $\stsp^o( \GL_n)$ coincides with the stack of very good splittings in \cite[\S3.3]{dCGZ}.

    The rest of the subsection is dedicated to the proof of \Cref{lemma: constant pi0}, which will be needed to prove that the stack of very good $G$-splittings is a torsor. We approach this by considering the sheaf $\pi_0(\stsp(G))$ of connected components. The following lemma shows that $\pi_0(\stsp(G))$ makes sense as a sheaf of sets.
\begin{lemma}
\label{lemma: torsor quot}
    Let $Q$ be a Picard full subcategory of a Picard category $X.$
    Let $H$ be a groupoid on which $X$ acts simply transitively as defined in \cite[\S3.6]{donagi2002gerbe}.
    The quotient 2-category $H/Q$ is equivalent to a set $\wt{H/Q}.$
    Furthermore, the action of $X$ on $Q$ induces a simply transitive action of the abelian group $\wt{X/Q}$ (as in \Cref{lemma: quot sub pic}) on $\wt{H/Q}.$
\end{lemma}
\begin{proof}
    The criterion \Cref{lemma: quot set} is satisfied by the definition of a simply transitive action in \cite[\S3.6]{donagi2002gerbe}.
    The first statement then follows.
    The transitivity of the action of $\wt{X/Q}$ on $\wt{H/Q}$ follows immediately from that of $X$ on $H.$ It remains to show the freeness.
    An element $\wt{h}$ in $\wt{H/Q}$ represents an isomorphism class of objects in $H/Q.$
    By the definition of 2-categorical quotients, two objects $h_1$ and $h_2$ are equivalent in $H/Q$ if there is an object $q$ and a morphism $qh_1\to h_2.$
    Let $\wt{x}$ be an element of $\wt{X/Q}$ that fixes $\wt{h}.$
    Then $\wt{x}$ (resp. $\wt{h}$) admits a lift to an object $x$ (resp. $h$) in $X/Q$ (resp. $H/Q$) such that there is an object $q$ in $Q$ and a  morphism $qh\to xh.$
    Since $H$ is a torsor, we have that there is a morphism $q\to x$ in $X.$ Therefore, $\wt{x}$ is trivial as desired. 
\end{proof}

\begin{defn} \label{defn: group of components g-splittings}
    We define $\pi_0(\stsp(G))$ to be the sheaf of sets on the big \'etale site of $A(G,\omega_{C'})$ associated to the quotient functor $\stsp(G)/\stP^o(\omega_{C'})$ as in \Cref{lemma: torsor quot}.
\end{defn}

\begin{lemma} \label{lemma: properties of components of splittings} \;
    \begin{enumerate}
        \item $\pi_0(\stsp(G))$ is a torsor under $\pi_0(\stP(\omega_{C'}));$
        \item For any geometric point $b'$ of $A(G,\omega_{C'}),$ the restriction $\pi_0(\stsp(G))|_{b'}$ corresponds to the set of connected components of the smooth stack $\stsp(G)_{b'}$ over $b'$.
        \item  Let $U$ be a scheme over $A(G,L)$ and let $s,s' \in \pi_0(\stsp(G))(U)$. Then $s = s'$ if and only if for all geometric points $b$ of $U$, we have equality of pullbacks $s|_{b} = s'|_{b}$.
    \end{enumerate}
\end{lemma}
\begin{proof}
    Parts (1) and (2) are direct consequences of \Cref{defn: group of components g-splittings}, \Cref{cztech}(2) and \Cref{lemma: torsor quot}. Part (3) follows from a similar argument as in \Cref{lemma: geometric points vs sheaf}, using that $\stsp(G)$ is a $\stP(\omega_{C'})$-torsor.
\end{proof}

\begin{defn}
    Let $\stP(\omega_{C'})^p \to A(G,\omega_{C'})$ denote the smooth group stack defined by the fiber product $\stP(\omega_{C}^{\otimes p})\times_{A(\omega_{C}^{\otimes p}), Fr^*} A(\omega_{C'})$.
    We denote by $\stP^o(\omega_{C'})^p$ the corresponding open substack of neutral components in $\stP(\omega_{C'})^p$. 
    We denote by $\pi_0(\stP(\omega_{C'})^p)$ the sheaf of abelian groups in the big \'etale site of $A(G,\omega_{C'})$ associated to the quotient functor $\stP(\omega_{C'})^p/\stP^o(\omega_{C'})^p$.
\end{defn}

Frobenius pullback induces a homomorphism of group stacks $Fr^*: \stP(\omega_{C'}) \to \stP(\omega_{C'})^p$ over $A(G,\omega_{C'})$, which necessarily sends  $\stP^o(\omega_{C'})$ into $\stP^o(\omega_{C'})^p$. On the other hand, the forgetful morphism induces a map $\stsp(G) \to \stP(\omega_{C'})^p$ that is equivariant with respect to the actions of $\stP^o(\omega_{C'})$. Hence, the forgetful morphism induces a well-defined morphism of sheaves of quotients $\pi_0(forget):\pi_0(\stsp(G))\to \pi_0(\stP(\omega_{C'})^p)$.

\begin{prop}
\label{prop: forget is inj}
    The morphism $\pi_0(forget):\pi_0(\stsp(G))\to \pi_0(\stP(\omega_{C'})^p)$ of sheaves of sets on the Hitchin base $A(G,\omega_{C'})$ induced by the forgetful morphism is injective. 
\end{prop}
\begin{proof}
    By \Cref{lemma: properties of components of splittings}(3) and \Cref{lemma: geometric points vs sheaf}, we can reduce to a statement of restrictions to geometric points of $A(G,\omega_{C'})$. Choose a geometric point $b'$ of $A(G, \omega_{C'})$. Let $(E,\nabla)$ and $(E',\nabla')$ be two objects in $\stsp(G)_{b'}$ 
    such that $E$ and $E'$ lie in the same connected component of $\stP(\omega_{C}^{\otimes p})_{b^p}.$
    By \Cref{cztech}(2), there is a unique (up to isomorphism) object $F$ in $\stP(\omega_{C'})_{b'}$ that sends $(E,\nabla)$ to $(E',\nabla')$. We are done if we can show that $F$ is in $\stP^o(  \omega_{C'})_{b'}.$
    However, twisting by the inverse object $E^{-1}$ of $E,$ we have an isomorphism $Fr^*F\cong E'\times^{J_{b^p}} E^{-1}$ in $\stP^o(\omega_C^{\otimes p})_{b^p}.$ By the injectivity of Frobenius pullback on connected components (\Cref{coroll: injectivity frobenius pullback}(2)), we conclude that $F$ is in $\stP^o( \omega_{C'})_{b'}$, as desired.
\end{proof}

\begin{lemma}
\label{lemma: constant pi0}
    The sheaf $\pi_0(\stsp(G))$ is constant over any $\bG_m$-orbit in the Hitchin base $A( G, \omega_{C'}).$ In particular, the image of $\stsp^o(G) \to A(G,\omega_{C'})$ is preserved by the $\mathbb{G}_m$-action on the Hitchin base.
\end{lemma}
\begin{proof}
 By \Cref{lemma: properties of components of splittings} and \Cref{prop: forget is inj}, we have that $\pi_0(\stsp(G))$ is a $\pi_0(\stP(\omega_{C'}))$-torsor which is also a subsheaf of $\pi_0(\stP(\omega_{C'})^p)).$ Note that for any line bundle $L$ we have that $\mathcal{P}(L) \to A(G,L)$ is
$\mathbb{G}_m$-equivariant by the $\mathbb{G}_m$-equivariance of $J_L$. In particular $\stP(\omega_{C'}) \to A(G,\omega_{C'})$ and $\stP(\omega_{C'})^p \to A(G,\omega_{C'})$ are $\bG_m$-equivariant, and hence we have that over any $\bG_m$-orbit of $A(G,\omega_{C'}),$ both $\pi_0(\stP(\omega_{C'}))$ and $\pi_0(\stP(\omega_{C'})^p)$ are constant.
    It follows that $\pi_0(\stsp(G))$ has to be also constant over the $\bG_m$-orbit. 
    
    % We are left to show that the image of $\stsp^o(G) \to A(G,\omega_{C'})$ is preserved by the $\mathbb{G}_m$-action. Let $x$ be a geometric point of $ \stsp^o(G)$, and let $b$ denote the image of $x$ in $A(G,\omega_{C'})$. Consider the orbit $\mathbb{G}_m \cdot b \to A(G,\omega_{C'})$. Since the pullback $\pi_0(\stsp(G))|_{\mathbb{G}_m \cdot b}$ is constant, the image of $x$ in $\pi_0(\stsp(G))|_b$ lifts uniquely to a section $s$ of $\pi_0(\stsp(G))|_{\mathbb{G}_m \cdot b}$. For each geometric point $b'$ in the orbit $\mathbb{G}_m \cdot b$, the restriction $s|_{b'}$ comes from a point $x'$ in $\stsp(G)_{b'}$. Note that, under the morphism of locally constant sheaves of sets $\pi_0(forget): \pi_0(\stsp(G))|_{\mathbb{G}_m \cdot b} \to \pi_0(\stP(\omega_{C'})^p)|_{\mathbb{G}_m \cdot b}$, the section $s$ maps to the zero section, since this holds over $b$ by construction.
    % It follows that $x'$ maps to the neutral component of $\stP(\omega_{C'})^p_{b'}$, and therefore it is a geometric point of $\stsp^o(G)_{b'}$. Therefore, the image of $\stsp^o(G)$ contains the image of the orbit $\mathbb{G}_m \cdot b$, as desired.
    
\end{proof}

\subsection{The stack of very good \texorpdfstring{$G$}{G}-splittings is a torsor.} \label{subsection: torsor}\quad

Recall that $h(G)$ is defined to be the Coxeter number of $G$ as in \cite[\S5.1]{serre2005complete}.

\begin{prop}[Smooth pseudo-torsor]\label{surj}
    The $A(G,\omega_{C'})$-stack $\stsp^o(G)$ is an open substack of $\stsp(G)$ that is smooth over $A(G,\omega_{C'})$.
    If $p\geq h(G)$, then the $A(G,\omega_{C'})$-stack $\stsp^o(G)$ is a torsor under the Picard stack $\stP^o(\omega_{C'}).$
\end{prop}
\begin{proof}
Openness follows from the openness of $\stP^o(\omega_{C}^{\otimes p})$ in $\stP(\omega_{C}^{\otimes p}).$
    Smoothness follows from \Cref{cztech}(1).
    
    We now show that $\stsp^o(G)$ is pseudo-torsor under $\stP^o(\omega_{C'})$. Let $(E,\nabla)$ and $(E',\nabla')$ be two objects in $\stsp^o(G)_{b'}$ for some geometric point $b'$ of $A(G, \omega_{C'}).$
    By \Cref{cztech}(2), there is a unique (up to isomorphism) object $F$ in $\stP^o( \omega_{C'})$ such that $F\cdot (E,\nabla)\cong (E',\nabla').$
    In particular we have that $E'\cong Fr^*F\times^{J_{b^p}} E.$
    Twisting by the inverse object of $E,$ we are reduced to showing that Frobenius pullback induces an injection $Fr^*: \pi_0(\stP(  \omega_{C'})_{b'})\to \pi_0(\stP( \omega_C^{\otimes p})_{b^p})$, which is the content of \Cref{coroll: injectivity frobenius pullback}.

     It remains to show that the morphism $\stsp^o(G)\to A(G,\omega_{C'})$ is surjective. 
     Since $\stsp^o(G)\to A(G,\omega_C')$ is smooth, it follows that the image of $\stsp^o(G)$ is open. This image is $\mathbb{G}_m$-equivariant by \Cref{lemma: constant pi0}. Furthermore, it contains the origin $0'$ of the Hitchin base by \Cref{lemma: 0-fiber G} proven below. Hence, the image must be the whole $A(G,\omega_{C'})$, as desired.
\end{proof}

The rest of this subsection is devoted to showing the necessary \Cref{lemma: 0-fiber G} in the proof of \Cref{surj}.
In the following, we denote by $0'$ and $0^p$ the origins of the Hitchin bases $A( G,\omega_{C'})$ and $A( G, \omega_{C}^{\otimes p})$ respectively.

\begin{lemma}
\label{sl2}
    The fiber category $\stsp^o(SL_2)_{0'}$ is non-empty.
\end{lemma}
\begin{proof}
Let $TC' \times_{C'} C=: T^pC$.
    By matrix calculation, we have that $J_{PGL_2,0^p}\cong T^pC$ and that $J_{SL_2,0^p}\cong T^pC\times\mu_2,$ where $TC'$ is understood as the vector group scheme underlying the tangent bundle of $C'$ (see also \cite[Lem. 3.21 and above Rem. 3.24]{chen-zhu}).
    Given a $W_2(k)$-lift $C_1$ of $C,$ \cite[p.251 (c)]{deligne1987relevements} entails that the sheaf $N$ of Frobenius lifts of $C$ is an $T^pC$-torsor. Hence, it defines an extension of commutative group schemes
    \begin{equation}
    \label{eqn: ext}
        0\to T^pC\to \mathcal{E}\xrightarrow{\pi} \bG_a\to 0,
    \end{equation}
    where $N$ is identified with $\pi^{-1}(1).$
    It is shown in \cite[Thm. 4.5]{ov07} and reinterpreted in \cite[\S3.5]{chen-zhu} that (\ref{eqn: ext}) is indeed an extension of flat connections, where $T^pC$ and $\bG_a$ are given the Cartier connections. We denote by $\nabla^{\mathcal{E}}$ the connection on $\mathcal{E}$. Since 1 is a horizontal section in $\bG_a,$ it follows that $\nabla^{\mathcal{E}}$ restricts to a flat connection $\nabla^N$ on $N.$ We thus have a $J_{PGL_2,0^p}$-flat connection $(N,\nabla^N).$
    It is also shown in \cite[\S3.5]{chen-zhu} that $\pi|_N$ coincides with the $p$-curvature morphism
    \[\pi|_N=\psi(\nabla^N): N\to Lie(J_{PGL_2,0^p})\otimes \omega_C^{\otimes p}\cong \bG_a.\]
    A direct calculation shows that $\tau(0)=1$ for $PGL_2,$ thus $(N,\nabla^N)$ is an object in $\stsp( PGL_2)_{0'}.$

    In order to address the group $SL_2$, replace (\ref{eqn: ext}) with 
    \begin{equation}
    \label{eqn: etx2}
        0\to T^pC\times \mu_2 \to \mathcal{E}\times \mu_2\xrightarrow{\pi\circ pr_{\mathcal{E}}} \bG_a\to 0.
    \end{equation}
    It is still an extension of flat connections.
    The $J_{SL_2, 0^p}$-flat connection $(N\times\mu_2,\nabla^{N\times \mu_2})$ lies in $\stsp(SL_2)_{0'}.$ 
    Furthermore, (\ref{eqn: etx2}) defines a deformation from $N\times \mu_2$ to the trivial torsor $T^pC\times\mu_2.$
    Therefore, we see that $(N\times \mu_2,\nabla^{N\times \mu_2})$ indeed lies in $\stsp^o(SL_2)_{0'}.$
\end{proof}

\begin{lemma}
\label{lemma: 0-fiber G}
    If $p\geq h(G)$, then the fiber $\stsp^o(G)_{0'}$ is non-empty.
\end{lemma}
\begin{proof}
We follow the construction at the end of \cite[\S3.5]{chen-zhu}.
The Kostant section takes the origin $0$ of $\fc$ to a regular nilpotent element $\kappa(0)$ in $\fg.$
     By \cite[Prop. 2]{serre1996exemples}, there is a principal $\varphi: SL_2\to G$ such that $d\varphi(e) = \kappa(0)$ for $e = \begin{bmatrix} 0 & 1 \\ 0 & 0 \end{bmatrix} \in Lie(SL_2)$.
    We have that $\varphi$ restricts to a morphism $J_{SL_2,0}\to J_{G,0}$ of $k$-group schemes. Via twisting and Frobenius pullback, we have a morphism of $C$-group schemes with flat connections $\varphi^p: (J_{SL_2,0^p},\nabla^{can})\to (J_{G,0^p},\nabla^{can}).$
    Therefore, change of groups via $\varphi^p$ gives a morphism of stacks $\varphi^p_*: \stdR(J_{SL_2})_{0'}\to \stdR(J_{G})_{0'}.$ By definition of the tautological section $\tau$ and $\varphi,$ we have that $d\varphi: Lie(J_{SL_2,0})\to Lie(J_{G,0})$ sends $\tau_{SL_2}(0)$ to $\tau_G(0).$ Therefore, the change of groups $\varphi^p_*$ restricts to a morphism of stacks $\stsp(SL_2)_{0'} \to \stsp(G)_{0'}.$
    Since $\varphi^p_*$ takes trivial $J_{SL_2,0^p}$-torsors to trivial $J_{G,0^p}$-torsors, we have that $\varphi^p_*$ further restricts to $\stsp^o(SL_2)_{0'}\to \stsp^o(G)_{0'}.$ We can now conclude using \Cref{sl2}.
\end{proof}

\section{Semistable Non Abelian Hodge Theorem}
 In \S\ref{subsection: semistability and moduli spaces} we establish some technical lemmas about stability of Higgs bundles and the moduli spaces appearing in \Cref{thm: main thm intro}.
 In \S\ref{sec: ss}, we prove our main Semistable NAHT (\Cref{thm: main thm intro}). In \S\ref{sec: dt}, we establish the isomorphic decomposition theorems for the Hitchin and the de Rham-Hitchin morphisms (\Cref{thm: intro thm dt}). 
 
\subsection{Stability and moduli spaces} \label{subsection: semistability and moduli spaces} \;

We denote by $\stHig^{ss}(C',G,\omega_{C'}) \subset \stHig(C',G,\omega_{C'})$ and $\stdR^{ss}(C,G) \subset \stdR(C,G)$ the open substacks of semistable objects.

Recall that the group $\pi_1(G):=X_*(T)/X_{coroots}$ canonically parametrizes the connected components of the stack of $G$-bundles on $C$ \cite[Thm. 5.8]{hoffmann-connected-components}. Given a $G$-bundle $E$ lying in the component labelled by $d\in \pi_1(G)$, we call $d$ the degree of $E$.
From now on we sometimes add the decoration $d\in \pi_1(G)$ in the moduli stacks/spaces, indicating that we only take the part whose underlying $G$-bundle has degree $d$.

\begin{prop}[{\cite[Thm. 2.26]{hererro-zhang}}] \label{prop: hz moduli space}    Suppose that $p>2h(G)-2$.
    Then $\stHig^{ss}(C',G,\omega_{C'})$ and $\stdR^{ss}(C,G)$ admit adequate moduli spaces $\schHig(C')$ and $\schdR(C)$ respectively. 
     For any given degree $d \in \pi_1(G)$, the induced adequate moduli spaces $\schHig(C', d)$ and $ \schdR(C, d)$ are quasi-projective schemes. \qed
\end{prop} 

We will need the following technical lemma.
\begin{lemma} \label{lemma: locally reductive}
    Assume that $p>2h(G)-2$. Then the stack $\stHig^{ss}(C',G,\omega_{C'})$ is locally reductive as in \cite[Def. 2.5]{alper2018existence}.
\end{lemma}
\begin{proof}
    Since $\stHig^{ss}(C',G,\omega_{C'})$ admits an adequate moduli space which is a disjoint union of finite type schemes over $k$, it follows that every point specializes to a closed point. To conclude, we shall show that $\stHig^{ss}(C',G,\omega_{C'})$ admits a Zariski cover by quotient stacks of the form $[\Spec(A)/H]$ where $H$ is a reductive group. Let $\schHig(C') = \bigcup_i U_i$ be an affine open cover of the scheme $\schHig(C')$, and denote by $\mathcal{U}_i \subset \stHig^{ss}(C',G,\omega_{C'})$ the preimage of $U_i$. Then $\mathcal{U}_i \subset \stHig^{ss}(C',G,\omega_{C'})$ is an open substack that admits an affine adequate moduli space $U_i$, and therefore it is quasi-compact. We shall conclude by showing that $\mathcal{U}_i$ is of the form $[Spec(A)/H]$ for some reductive group $H$. For any given set $\{x_1, x_2, \ldots, x_n\}$ of distinct $k$-points of $C$, we will denote by $\stHig^{fr}(C',G, \omega_{C'}) \to \stHig(C',G, \omega_{C'})$ the stack of framed Higgs bundles, which parametrizes Higgs $G$-bundles $(E,\phi)$ along with the extra structure of a trivialization of the restriction $E|_{x_j}$ for each $x_j$. The group $\prod_{j=1}^n G$ acts naturally on $\stHig^{fr}(C',G, \omega_{C'})$ by changing the trivialization at each point $x_j$, and the affine morphism $\stHig^{fr}(C',G, \omega_{C'}) \to \stHig(C',G, \omega_{C'})$ exhibits $\stHig^{fr}(C',G, \omega_{C'})$ as a $\prod_{j=1}^n G$-torsor over $\stHig(C',G, \omega_{C'})$. We set $\mathcal{U}_i^{fr}:= \stHig^{fr}(C',G, \omega_{C'}) \times_{\stHig(C',G, \omega_{C'})} \mathcal{U}_i$. We have again an affine morphism $\mathcal{U}_i^{fr} \to \mathcal{U}_i$ which exhibits $\mathcal{U}_i^{fr}$ as a $\prod_{j=1}^nG$-torsor over $\mathcal{U}_i$. The proof of \cite[Prop. 5.4.1.3]{gaitsgory-lurie} (using the fact that the union of all closed points of the curve $C'$ is scheme-theoretically dense inside $C'$, and this remains true after base-changing to any $k$-scheme $S$) applied to the quasi-compact stack $\mathcal{U}_i$ implies that, after perhaps enlarging the number of points $\{x_1, x_2, \ldots, x_n\}$, we may assume that the objects in $\mathcal{U}_i^{fr}$ don't have any automorphisms. Therefore, $\mathcal{U}_i^{fr}$ is an algebraic space. The composition $\mathcal{U}^{fr}_i \to \mathcal{U}_i \to U_i$ of an affine morphism and an adequately affine morphism is adequately affine \cite[Prop. 4.2.1(1)]{alper_adequate}. By \cite[Thm. 4.3.1]{alper_adequate}, it follows that the morphism $\mathcal{U}^{fr}_i \to U_i$ is affine, and so in particular the algebraic space $\mathcal{U}^{fr}_i$ is an affine scheme $\Spec(A)$. Since $\Spec(A) = \mathcal{U}^{fr}_i \to \mathcal{U}_i$ is a $\prod_{j=1}^nG$-torsor, it follows that $\mathcal{U}_i = [\Spec(A)/\prod_{j=1}^nG]$, as desired.
\end{proof}

In this paper, we also need the notion of stability.
\begin{defn}[Stable Higgs bundles] \label{defn: stable higgs bundles}
    Let $(E, \phi)$ be a geometric point of $\stHig(C',G, \omega_{C'})$ defined over an algebraically closed field $K \supset k$. We say that $(E,\phi)$ is stable if for all strictly smaller parabolic subgroups $P\subsetneq G_K$, all $\phi$-compatible reductions of structure group $E_{P} \subset E$ (as in \cite[Def. 2.12]{hererro-zhang}), and all $P$-dominant characters $\chi$ (as in \cite[Def. 2.15]{hererro-zhang}), we have $\text{deg}(E_{P}(\chi)) <0$ for the degree of the associated line bundle $E_{P}(\chi)$.
\end{defn}

\begin{prop} \label{lemma: stable points are saturated}
    Suppose that $p>2h(G)-2$. Let $x$ be a stable geometric point of $\stHig^{ss}(C')$. Let $y$ denote its image in $\schHig(C')$. Then $x$ is a closed point of the fiber $\stHig^{ss}(C')_{y}$.
\end{prop}
\begin{proof}
    The point $y = \Spec(K)$ is defined over some algebraically closed over-field $K \supset k$. Suppose for the sake of contradiction that $x$ is not a closed point of $\stHig^{ss}(C', G,  \omega_{C'})_{y}$, so it specializes to a distinct closed $K$-point $z$ in $\stHig^{ss}(C',\omega_{C'})_{y}$. The fiber $\stHig^{ss}(C', G,  \omega_{C'})_{y}$ admits an adequate moduli space that is finite over $\Spec(K)$ \cite[Prop. 5.2.9(3) + Thm. 6.3.3]{alper_adequate}, and so every $K$-point specializes to a closed point. By base-changing the local quotient stack presentation from \Cref{lemma: locally reductive}, it follows that $\stHig^{ss}(C', G,  \omega_{C'})_{y}$ is locally reductive.
    By the Hilbert-Mumford criterion \cite[Lem. 3.24]{alper2018existence}, there is a morphism $f: \Theta_K := [\mathbb{A}^1_K/\mathbb{G}_m] \to \stHig^{ss}(C', G,  \omega_{C'})_{y}$ such that $f(0) \cong z$ and $f(1) \cong x$. By \cite[Prop. 4.7]{hererro-zhang}, the morphism $f$ corresponds to a $\phi$-compatible weighted parabolic reduction $(\lambda, E_{P_{\lambda}})$ of the point $x=(E, \phi)$. Since $z$ is distinct from $x$, the cocharacter $\lambda$ does not land in the center of $G_K$, and the associated parabolic subgroup $P_{\lambda} \subset G_K$ is strictly smaller than $G_K$. Choose a $P_{\lambda}$-dominant character $\chi$. Then, by stability of $x$, we have $\text{deg}(E_{P_{\lambda}}(\chi))<0$. 
    
    The point $z$ corresponds to the associated Levi bundle $(E_{L_{\lambda}}, \phi_{L_{\lambda}})$, and so it admits two canonical weighted parabolic reductions $(\lambda, E_{P_{\lambda}})$ and $(-\lambda, E_{P_{-\lambda}})$. Notice that $-\chi$ is $P_{-\lambda}$-dominant, and we have the following inequality contradicting the semistability of $z$: 
    \[\text{deg}(E_{P_{-\lambda}}(-\chi)) = \text{deg}(E_{L_{\lambda}}(-\chi)) = -\text{deg}(E_{L_{\lambda}}(\chi)) = -\text{deg}(E_{P_{\lambda}}(\chi))>0.\]
\end{proof}

\begin{coroll} \label{cor: stable substack is saturated}
    Suppose that $p> 2h(G)-2$. If $\mathcal{U} \subset \stHig^{ss}(C',G, \omega_{C'})$ is an open substack contained in the locus of stable geometric points, then $\mathcal{U}$ is saturated with respect to the moduli space morphism $\stHig^{ss}(C',G,\omega_{C'}) \to \schHig(C')$.
\end{coroll}
\begin{proof}
    Choose a geometric point $y$ of $\schHig(C')$ such that $\mathcal{U}_y$ is nonempty. We need to show the equality of fibers $\mathcal{U}_y = \stHig^{ss}(C',G, \omega_{C'})_y$. Since $\stHig^{ss}(C',G, \omega_{C'})_y$ admits an adequate moduli space which is universally homeomorphic to the point $y$ (\cite[Prop. 5.2.9(3]{alper_adequate}), it follows that $\stHig^{ss}(C',G, \omega_{C'})_y$ has a unique closed $y$-point and all other $y$-points specialize to it \cite[Thm. 5.3.1(5]{alper_adequate}. Now the image of the open substack $\mathcal{U}_y \subset \stHig^{ss}(C',G, \omega_{C'})_y$ is nonempty, closed under generalization, and all of its $y$-points are closed by \Cref{lemma: stable points are saturated}. Therefore, we must have that $\stHig^{ss}(C',G, \omega_{C'})_y$ consists of a single geometric point and $\mathcal{U}_y = \stHig^{ss}(C',G, \omega_{C'})_y$.
\end{proof}

Next, we construct moduli spaces for $\stP^o(\omega_{C'})$ and $\stsp^o(G)$. We denote by $I_{\stP(\omega_{C'})} \to \stP(\omega_{C'})$ the inertia group stack, which is relatively affine over $\stP(\omega_{C'})$. The containment of $Z_G \times [\mathfrak{g}/G]$ inside the regular centralizer $J$ induces a natural inclusion $Z_G \times \stP(\omega_{C'}) \hookrightarrow I_{\stP(\omega_{C'})}$ of relatively affine group stacks over $\stP(\omega_{C'})$.

\begin{prop} \label{prop: group scheme moduli}
    The following statements hold:
    \begin{enumerate}[1.]
        \item The inclusion $Z_G \times \stP(\omega_{C'}) \hookrightarrow I_{\stP(\omega_{C'})}$ is an isomorphism.
        \item The stack $\stP^o(\omega_{C'})$ admits a good moduli space $\stP^o(\omega_{C'}) \to \schP^o(C')$, and $\schP^o(C')$ is a smooth commutative group algebraic space with geometrically connected fibers over $A(G,\omega_{C'})$.
        \item The stack $\stsp^o(G)$ admits a good moduli space $\stsp^o(G) \to \schsp^o(C)$, and $\schsp^o(C) \to A(G,\omega_{C'})$ is an \'etale torsor for the group algebraic space $\schP^o(C',\omega_{C'})$.
    \end{enumerate}
\end{prop}
\begin{proof}
    Part (1) follows from \cite[Prop. 4.17(i)]{decomposition-higgs}. Part (2) is \cite[Prop. 4.17(ii)]{decomposition-higgs}. Part (3) follows directly from the fact that $\stsp^o(G)$ is an \'etale torsor for $\stP(\omega_{C'})$, by the compatibility of the formation of good moduli spaces with base-change \cite[Prop. 4.7(i]{alper-good-moduli}.
\end{proof}

\begin{lemma} \label{lemma: action preserves stability}
  The action of the group stack $\stP^o(\omega_{C'})$ preserves the locus of stable geometric points inside $\stHig(C',G,\omega_{C'})$.
\end{lemma}
\begin{proof}
    Let us first give another description of the notion of stability. Let $(E, \phi)$ be a $k$-point of $\stHig(C',G,\omega_{C'})$. By \cite[Prop. 4.7]{hererro-zhang}, $\phi$-compatible weighted parabolic reductions $(\lambda, E_{P_{\lambda}})$ (consisting of a cocharacter $\lambda: \mathbb{G}_m \to G$ and a $\phi$-compatible reduction of structure group to $P_{\lambda}$) are in natural correspondence with morphisms of stacks $f: \Theta_k := [\mathbb{A}^1_k/\mathbb{G}_m] \to \stHig(C',G,\omega_{C'})$ along with $f(1) \cong (E, \phi)$. We call a compatible weighted parabolic reduction $(\lambda, E_{P_{\lambda}})$ \textit{central} if the image of $\lambda$ is contained in the center $Z_G \subset G$, and otherwise we say that $(\lambda, E_{P_{\lambda}})$ is non-central. Similarly, under the natural correspondence, we obtain the notion of central and non-central morphisms $f: \Theta_k  \to \stHig(C',G,\omega_{C'})$. By definition, the parabolic subgroup $P_{\lambda}$ is strictly smaller than $G$ if $(\lambda, P_{\lambda})$ is central. The same considerations as in \cite[\S 4.1]{hererro-zhang} applied to \Cref{defn: stable higgs bundles} show that there is a fixed line bundle $\mathcal{L}:= \mathcal{D}(\mathfrak{g})$ on the stack $\stHig(C',G,\omega_{C'})$ such that a Higgs bundle $(E, \phi)$ is stable if and only if for all non-central $f:\Theta_k \to \stHig(C',G,\omega_{C'})$ with $f(1) \cong (E, \phi)$, the $\mathbb{G}_m$-weight of the fiber $f^*(\mathcal{L})|_0$ is strictly negative.

    Let us give a stacky interpretation of non-central $f: \Theta_k \to \stHig(C',G,\omega_{C'})$. Note that the inertia stack of $\stHig(C',G,\omega_{C'})$ contains a central copy of the constant relative group scheme $Z_G \times \stHig(C',G,\omega_{C'}) \to \stHig(C',G,\omega_{C'})$. Hence, we may rigidify in the sense of \cite[Def. 5.1.9]{acv-twisted-bundles-covers} to obtain a stack $\stHig(C',G,\omega_{C'}) \to \stHig(C',G,\omega_{C'})\hollowslash Z_G$. A morphism $f: \Theta_k \to \stHig(C',G,\omega_{C'})$ is central if and only if the composition 
    \[\Theta_k \xrightarrow{f} \stHig(C',G,\omega_{C'}) \to \stHig(C',G,\omega_{C'})\hollowslash Z_G\]
    factors through a point $\Spec(k) \to \stHig(C',G,\omega_{C'})\hollowslash Z_G$. 
        Taking $Z_G$-rigidifications for the action morphism $\stP^o(\omega_{C'}) \times_{A(G,\omega_{C'})} \stHig(C',G,\omega_{C'}) \to \stHig(C',G,\omega_{C'})$ induces an action of the moduli space $\schP^o(C)$ (see \Cref{prop: group scheme moduli}(2)) on $\stHig(C',G,\omega_{C'})\hollowslash Z_G$.

        Let $x$ and $g$ be geometric points of $\stHig(C',G,\omega_{C'})$ and $\stP^o(G,\omega_{C'})$ respectively, with the same image $b$ in $A(G,\omega_{C'})$. After base-change, we may assume without loss of generality that they are $k$-points. We want to show that $x$ is a stable if and only if $g \cdot x$ is stable. 
        The action of $\stP^o(G,\omega_{C'})$ induces an orbit morphism $\stP^o(G,\omega_{C'})_b \to \stHig(C',G,\omega_{C'})_b$ such that $x$ is the image of the identity and $g\cdot x$ is the image of $g$. 
        Any morphism $f: \Theta_k \to \stHig(C',G,\omega_{C'})$ with $f(1) \cong x$ necessarily factors through $\stHig(C',G,\omega_{C'})_b$, and the action induces a morphism $\widetilde{f}: \Theta_k \times \stP^o(G,\omega_{C'})_b \to \stHig(C',G,\omega_{C'})_b$ such that the restriction to $\Theta_k \times g$ satisfies $\widetilde{f}|_{\Theta_k \times g}(1) \cong g \cdot x$. 
        This establishes a bijection $f \mapsto g \cdot f$ between testing morphisms $\Theta_k \to \stHig(C',G,\omega_{C'})$ for $x$ and for $g\cdot x$. Note that $f$ is (non)central if and only if $g\cdot f$ is (non)central. 
        Since the $\mathbb{G}_m$-weight of the pullback of $\mathcal{L}$ under the restriction $\widetilde{f}|_0: 0\times \stP^o(G,\omega_{C'})_b \to \stHig(C',G,\omega_{C'})_b$ is locally constant and $\stP^o(G,\omega_{C'})_b$ is connected, the weight of $f^*(\mathcal{L})|_0$ agrees with that of $(g \cdot f)^*(\mathcal{L})|_0$. Hence, the stability condition for the weight on testing morphisms $\Theta_k \to \stHig(C',G,\omega_{C'})$ for $x$ is equivalent to the stability condition for $g \cdot x$, as desired.
 
\end{proof}

 \subsection{Semistable Nonabelian Hodge correspondence in positive characteristic}\label{sec: ss}\;

\begin{thm}
\label{thm: ss main}
Assume that $p\geq h(G)$. 
\begin{enumerate}
    \item The Chen-Zhu isomorphism \eqref{cziso} restricts to an isomorphism of $A(G,\omega_{C'})$-stacks:
    \begin{equation}
    \label{eqn: ss1}
        \fC^{ss}: \stsp^o(G)\times^{\stP^o(\omega_{C'})} \stHig^{ss}(C',G,\omega_{C'})\xrightarrow{\sim} \stdR^{ss}(C,G).
    \end{equation}
    \item For each degree $d\in \pi_1(G)$, the isomorphism \eqref{eqn: ss1} sends degree $d$ Higgs bundles to degree $pd$ connections.
    \item   
    If $p>2h(G)-2$, then for all $d \in \pi_1(G)$ the isomorphism \eqref{eqn: ss1} induces an isomorphism on the level of quasi-projective adequate moduli spaces:
    \begin{equation}
    \label{eqn: ss2}
        \mathbb{C}^{ss}: \schsp^o(C,G)\times^{\schP^o(C')} \schHig(C', d)\xrightarrow{\sim} \schdR(C, pd),
    \end{equation}
    where $\schP^o(C') \to A(G,\omega_{C'})$ is a smooth quasi-projective group scheme with geometrically connected fibers and $\schsp^o(C,G)$ is an \'etale torsor for $\schP^o(C')$.
\end{enumerate}
  
\end{thm}
\begin{proof}
Let us first prove 1.
Let $s\in \stsp^o(K)$ be a geometric point of $\stsp^o$ defined over some algebraically closed field $K.$
We denote by $b'$ its image in $A(\omega_{C'})(K)$.
We need to show that the isomorphism $s\times \stHig(C')_{b'}\xrightarrow{\sim} \stdR(C)_{b'}$ preserves semistability.
We shall use the theory of $\Theta$-semistability.
Let $\Bun_G(C)$ be the moduli stack of $G$-bundles on $C.$
We denote by $\mathcal{L}'$ (resp. $\mathcal{L}$) the determinant line bundle on $\Bun_G(C')$ (resp. $\Bun_G(C)$) as defined in \cite[\S 1.F.a]{heinloth-hilbertmumford}.
Let $\mathcal{L}^{Dol}$ (resp. $\mathcal{L}^{dR}$) be the line bundle on $\stHig(C')$ (resp. $\stdR(C)$) obtained via pullback from the forgetful morphism $\stHig(C')\to \Bun_G(C')$ (resp. $\stdR(C)\to \Bun_G(C)$).
Let $\Theta_K$ be the quotient stack $\mathbb{A}^1_K/\bG_m.$ It has an open schematic point $1\cong \bG_m/\bG_m.$
Let $x'$ be a $K$-point of $\stHig(C')$. 
By \cite[Prop. 4.9]{hererro-zhang}, the point $x'$ lies in the semistable locus $\stHig^{ss}(C')$ if and only if for all morphisms $f:\Theta_{K}\to \stHig(C')$ with $f(1)\cong x',$ the weight $wt(\mathcal{L}^{Dol})(f)$ of the $\bG_m$-action
on the 0-fiber of the equivariant line bundle $f^*\mathcal{L}^{Dol}\in Pic^{\bG_m}(\mathbb{A}^1_{K})$ is non-negative.
The same characterization holds for semistable points of $\stdR^{ss}(C)$, by considering the weights of the line bundle $\mathcal{L}^{dR}$ instead.
Therefore \eqref{eqn: ss1} follows from \Cref{lemma: technical lemma 1 semistability} proven below.

We now prove 2.
Let $b'$ be any geometric point of $A(\omega_{C'})$, which after base-change we may assume to be defined over the ground field $k$.
Since $\stsp^o_{b'}$ is connected, for each $d\in \pi_1(G),$ there is a well-defined $x(d)\in \pi_1(G)$ such that the morphism $\stsp^o_{b'}\times\stHig(d)_{b'}\to \stdR(x(d))_{b'}$ is well-defined.
By forgetting the connections, this morphism gives rise to a morphism $\Bun_{J_{b^p}}^o(C)\times \stHig(C',d)_{b'}\to \Bun_G(C, x(d)).$
Let $e$ be the trivial torsor in $\Bun_{J_{b^p}}^o(C).$
Since $\Bun_{J_{b^p}}^o(C)$ is connected, to determine $x(d),$ it suffices to look at the image of $f_e: e\times \stHig(C',d)_{b'}\to \Bun_G(C).$
Given any object $(E,\phi)$ in $\stHig(C',d)_{b'},$ we have that $f_e(E,\phi)=Fr^*E.$
In view of how the isomorphism $\pi_1(G)=\pi_0(\Bun_G(C))$ is established in \cite[Thm. 5.8]{hoffmann-connected-components}, we have that $x(d)=pd,$ i.e, for any geometric point $s\in \stsp^o$, the isomorphism $s\times \stHig(C')_{b'}\xrightarrow{\sim}\stdR(C)_{b'}$ restricts to $s\times \stHig(C', d)_{b'}\xrightarrow{\sim}\stdR(C,pd)_{b'}.$

Finally, we prove 3.    
By the universal property of adequate moduli spaces and their compatibility with flat base-change (\cite[Thm. 3.12]{alper2023etale} and \cite[Prop. 5.2.9(1)]{alper_adequate}), it follows that (\ref{eqn: ss1}) induces the desired isomorphism at the level of adequate moduli spaces. It remains to show that the good moduli space $\schP^o(C',\omega_{C'})$ is quasi-projective. 
By \Cref{lemma: kostant stable}, we have that the Kostant section $\kappa$ factors through the stable locus inside $\stHig^{ss}(C').$ The action of $\stP(\omega_{C'})$ on the Kostant section defines an open embedding $\stP(\omega_{C'})\hookrightarrow \stHig^{ss}(C',d)$
for some $d\in \pi_1(G)$ \cite[\S4.3]{ngo-lemme-fondamental}. Furthermore, this open embedding lies on the stable locus by \Cref{lemma: action preserves stability}. By \Cref{cor: stable substack is saturated}, it follows that the open substack $\stP(\omega_{C'}) \subset \stHig^{ss}(C',d)$ is saturated with respect to the adequate moduli space morphism $\stHig^{ss}(C',d) \to \schHig(C',d)$. Hence, the moduli space $\schP^o(C')$ of $\stP(\omega_{C'})$ is open inside the quasi-projective scheme $\schHig(C', d)$ (\Cref{prop: hz moduli space}). Therefore, $\schP^o(C')$ is quasi-projective.
\end{proof}

\begin{lemma} \label{lemma: technical lemma 1 semistability}
Let $s\in \stsp^o(K)$ be a geometric point of $\stsp^o$ defined over some algebraically closed field $K.$ Let $f:\Theta_{K}\to s\times \stHig(C',G,\omega_{C'})_{b'}$ be any morphism.
    Let $\wt{f}: \Theta_{K} \to \stdR(C,G)_{b'}$ be the composition of $f$ and the isomorphism $\fC^{ss}_s: s\times \stHig(C',G,\omega_{C'})_{b'}\xrightarrow{\sim} \stdR(C,G)_{b'}.$ Then we have $
        wt(\mathcal{L}^{dR})(\wt{f})= p\cdot wt(\mathcal{L}^{Dol})(f)$.
\end{lemma}
\begin{proof}
    Consider the following commutative diagram of $K$-stacks:
    \[
        \xymatrix{
        \Theta_K \ar[r]^-{f} \ar[dr]_-{f_s} & s\times \stHig(C',G,\omega_{C'})_{b'} \ar[r]^-{\fC^{ss}_s}_-{\sim} \ar[d] & \stdR(C)_{b'}\ar[d]\\
        & \Bun^o_{J_{b^p}}(C)\times \stHig(C',G,\omega_{C'})_{b'} \ar[r]_-{\psi} & \Bun_{G}(C),
        }
    \]
    where the vertical morphisms forget the connection, and $\psi$ is defined the same as $\fC^{ss}_s$ but we forget the connections, namely, given objects $E$ of $\Bun_{J_{b^p}}^o(C)$ and $(F,\phi)$ of $\stHig(C',G,\omega_{C'})_{b'}$, we define
$\psi(F, (E,\phi))$ to be $ F\times^{J_{b^p}}_{Fr^*a_{(E,\phi)}} Fr^*E$,
    where the morphism $a$ is as in \eqref{eqn: a}.

    The action of $\Bun_{J_{b^p}}^o(C)$ by left multiplication on the first factor of $\Bun_{J_{b^p}}^o(C)\times \stHig(C',G,\omega_{C'})_{b'}$ takes the $\Theta_K$-point $f_s$ to another $\Theta_K$-point $f_0$ of $\Bun_{J_{b^p}}^o(C)\times \stHig(C',G,\omega_{C'})_{b'}$ where the first factor gives the $\Theta$-family of trivial $J_{b^p}$-torsors on $C_{\Theta_K}.$
    Because $f_s$ and $f_0$ differ by an action of the connected Picard stack $\Bun_{J_{b^p}}^o(C)$ and the weight is a discrete invariant, we have the equalities
    $wt(\mathcal{L})(\psi\circ f_s)=wt(\mathcal{L})(\psi\circ f_0)$ and $
wt(\mathcal{L}^{Dol})(f_s) = wt(\mathcal{L}^{Dol})(f_0)$,
    where the $\mathcal{L}^{Dol}$ denotes to $\Bun_{J_{b^p}}^o(C)\times \stHig(C',G,\omega_{C'})_{b'}$ via the second projection.
        
    Consider the following commutative diagram of $K$-stacks:
    \[
        \xymatrix{
        \Theta_K \ar[r]^-{f_0} \ar[dr] \ar[ddr]_-{\wt{f_0}} & 0\times \stHig(C',G,\omega_{C'}) \ar[r]^-{\psi_0} \ar[d] & \Bun_G(C)\ar[d]^-{id}\\
        & \Bun_G(C') \ar[r]^-{Fr^*} \ar[d] & \Bun_G(C) \ar[d] \\
        & \Bun_{GL(\fg)}(C') \ar[r]^-{Fr^*} & \Bun_{GL(\fg)}(C),
        }
    \]
    where the 0 on the first row is the trivial $J_{b^p}$-torsor on $C;$
    $\psi_0$ is the restriction of $\psi,$ and on the level of $G$-bundles, $\psi_0$ is just given by Frobenius pullback; the bottom vertical arrows are given by taking adjoint bundles.

    Let $\mathcal{L}(C,GL(\fg))$ (resp. $\mathcal{L}(C',GL(\fg)))$ be the determinant line bundle on $Bun_{GL(\fg)}(C)$ (resp. $Bun_{GL(\fg)}(C)$) as defined in \cite[\S1.E.a]{heinloth-hilbertmumford}. By the explicit calculation of weight as in \cite[\S 1.E.c]{heinloth-hilbertmumford}, we have the numerical identity
    \[
        wt(\mathcal{L}(C,GL(\fg)))(Fr\circ \wt{f_0})=p\cdot wt(\mathcal{L}(C',GL(\fg)))(\wt{f_0}).
    \]
    Combining the numerical identities established above and the fact that the $\mathcal{L}$ on $\Bun_G(C)$ is just the pullback of $\mathcal{L}(GL(\fg))$ via taking the adjoint bundle, 
    we have the desired equality:
    \begin{align*}
        &wt(\mathcal{L}^{dR})(\wt{f})= wt(\mathcal{L})(\psi\circ f_s)
        = wt(\mathcal{L})(\psi\circ f_0) 
        = wt(\mathcal{L}(C,GL(\fg)))(Fr\circ \wt{f_0})  \\
        =& p\cdot wt(\mathcal{L}(C',GL(\fg)))(\wt{f_0})= p\cdot wt(\mathcal{L}^{Dol})(f_0)
        = p \cdot wt(\mathcal{L}^{Dol})(f_s)=p\cdot wt(\mathcal{L}^{Dol})(f).
    \end{align*}
\end{proof}

\begin{remark}
    A similar argument as in the proof of \Cref{thm: ss main} shows that the isomorphism \eqref{eqn: ss1} preserves the loci of stable points. This implies that the torsor $\schsp^o(C,G)$ is also a quasi-projective scheme, and that \eqref{eqn: ss2} restricts to an isomorphism $\mathbb{C}^{s}: \schsp^o(C,G)\times^{\schP^o(C')} \schHig^s(C', d)\xrightarrow{\sim} \schdR^s(C, pd)$ of moduli spaces of stable objects.
\end{remark}

 \subsection{Isomorphic Decomposition Theorems}\label{sec: dt}\;
 
In this subsection, we fix $d\in \pi_1(G)$ and consider the Hitchin $h_{Dol}:\schHig(C',d)\to A(G, \omega_{C'})$ and de Rham-Hitchin $h_{dR}:\schdR(C,pd)\to A(G, \omega_{C'})$ morphisms. If the characteristic $p>0$ of $k$ satisfies $p>2h(g)-2$, then these morphisms are proper by \cite[Thm. 5.20]{hererro-zhang}.
% \[h_{Dol}:\schHig(C',G,d)\to A(C'), \quad h_{dR}:\schdR(C,G,pd)\to A(C').\]

Choose a prime $\ell \neq p$. For any scheme $X$ of finite type over $k,$ let $D_c^b(X,\oql)$ be the bounded constructible derived category. All the pushforwards $f_*$ in this section are derived.

The following lemma was suggested to us by Sasha Petrov.

\begin{lemma}[Homotopy Lemma]
\label{lemma: homotopy}
    Let $f: X\to S$ be a morphism between two schemes of finite type over a field $k.$
Let $\pi: G\to S$ be a smooth group scheme with connected geometric fibers. 
Assume that there is an action of $G$ on $X$ relative to $S.$
Then the group of global sections $G(S)$ acts trivially on each cohomology sheaf $\mathcal{H}^i(f_*\oql).$
\end{lemma}
\begin{proof}
    If we replace $\mathcal{H}^i(f_*\oql)$ by the perverse cohomology sheaf $\pcs^i(f_*\oql),$ this lemma is \cite[Lem. 3.2.3]{laumon2008lemme}, which relies on the fact that the pullback by a smooth morphism with connected geometric fibers is fully faithful for perverse sheaves \cite[Prop. 4.2.5]{bbdg}. One proof of our current lemma is to note that the proof of \cite[Prop. 4.2.5]{bbdg} also works for the standard $t$-structure.
    An alternative proof is to use \cite[Lem 3.3.10]{behrend-thesis} to see that $G(S)$ acts trivially on $\mathcal{H}^i(f_*\mathbb{Z}/\ell^n)$, and then conclude by taking limits in $n$.
\end{proof}

  \begin{thm}\label{thm: dt} Suppose that $p>2h(G)-2$. Then:
     \begin{enumerate}[1]
         \item There is a canonical isomorphism of perverse cohomology sheaves in $D^b_c(A(G, \omega_{C'}),\oql):$
         \begin{equation*}
              \pcs^*(h_{Dol,*}\oql)\cong \pcs^*(h_{dR,*}\oql),\quad \pcs^*(h_{Dol,*}\IC)\cong \pcs^*(h_{dR,*}\IC).
         \end{equation*}
         \item We have a distinguished isomorphism in $D^b_c(A(G, \omega_{C'}),\oql):$
         \[
             h_{Dol,*}\IC\cong h_{dR,*}\IC.
         \]
         \item There is a distinguished isomorphism of intersection cohomology groups:
         \[
    I\!H^*(\schHig(C',d),\oql)\cong I\!H^*(\schdR(C,dp),\oql).
\]
Moreover, this isomorphism respects the perverse Leray filtrations induced by the Hitchin and the de Rham-Hitchin morphisms respectively.
\item We have canonical isomorphisms of cohomology sheaves:
    \[ \mathcal{H}^*(h_{Dol,*}\oql)\cong \mathcal{H}^*(h_{dR,*}\oql).\]
     \end{enumerate}
 \end{thm}
 Note that     \Cref{thm: dt}(4) is new even in the $G=GL_N$ case.

 \begin{proof}
     For (1), the proofs of \cite[Thms. 5.1, 5.2]{dCGZ} carry verbatim in our setting. We recall the argument for the reader's benefit. Using the isomorphism (\ref{eqn: ss2}), we obtain the isomorphisms in (1) \'etale locally over $A(G, \omega_{C'}).$
     We then glue the local isomorphisms together using: 
     \begin{enumerate}
         \item[i] The Homotopy Lemma \cite[Lem. 3.2.3]{laumon2008lemme};
         \item[ii] The fact that $\stP^o(C')$ has geometrically connected fibers over $A(G, \omega_{C'});$
         \item[iii] \cite[Prop. 3.2.2]{bbdg}, which entails that we can glue morphisms between two objects $K$ and $L$ in $D^b_c$ if $Ext^{<0}(K,L)=0.$
     \end{enumerate}
     
     (2) and (3) then follow from (1) and the Decomposition Theorem for perverse sheaves.
     There are several distinguished choices for the isomorphisms in the Decomposition Theorem, but none are canonical, hence the change of words from canonical to distinguished in the statement.

     The proof of (4) is the same as (1) except that we replace (i) with \Cref{lemma: homotopy}.
 \end{proof}

\begin{remark}
    When $G\neq GL_r$, the moduli spaces $\schHig$ and $\schdR$ are typically singular, no matter of the degrees.
    Therefore, at least a priori, the intersection cohomologies of them can be different from the $\oql$-cohomologies. 
\end{remark}

\appendix
\begin{section}{Stability of the Kostant section}
\label{sec: appendix}

In this section, we give an algebraic proof of the stability of the Kostant section (in the sense of \Cref{defn: stable higgs bundles}) under mild characteristic assumptions. As far as we know, the only proof in this direction is \cite[\S5]{hitchin-teichmuller}, which shows the polystability of the Kostant section over $\bC,$ relying on gauge-theoretic methods. 

\begin{lemma}[Openness of stability] \label{lemma: openness of stability}
   If $p\geq h(G)$, then the locus of stable geometric points in $\stHig^{ss}(C',G,\omega_{C'})$ is open. 
\end{lemma}
\begin{proof}
    We begin by showing that the locus of stable geometric points is closed under generalization. Let $R$ be a discrete valuation ring over $k$, and choose a morphism $j: \Spec(R) \to \stHig^{ss}(C',G,\omega_{C'})$. We want to show that if the image $j(\eta)$ of the generic point $\eta \in \Spec(R)$ is not stable, then the same holds for the image $j(s)$ of the special point $s \in \Spec(R)$. For this, we use the interpretation of stability in terms of morphisms $\Theta \to \stHig(C',G,\omega_{C'})$ explained in the proof of \Cref{lemma: action preserves stability}. Since $j(\eta)$ is not stable, after passing to a finite extension of $R$, we may assume that there exists a morphism $f: \Theta_{\eta} \to \stHig(C',G,\omega_{C'})$ with an isomorphism $f(1) \cong j(\eta)$ such that $wt(f^*(\cL)|_0)  \leq 0$. As explained in the proof of \Cref{thm: ss main}, the semistability of $j(\eta)$ implies that this weight is nonnegative, and therefore $wt(f^*(\cL)|_0)  = 0$. By \cite[Lem. 6.15]{alper2018existence}, it follows that the morphism $f: \Theta_{\eta} \to \stHig(C',\omega_{C'})$ factors through $\stHig^{ss}(C',G,\omega_{C'}) \subset \stHig(C',G,\omega_{C'})$. 
    Since $\stHig^{ss}(C',G,\omega_{C'})$ is locally reductive (\Cref{lemma: locally reductive}) and has an adequate moduli space, it follows that the stack $\stHig^{ss}(C',G,\omega_{C'})$ is $\Theta$-reductive by \cite[Thm. 5.4]{alper2018existence}. This means that we can extend $f: \Theta_{\eta} \to \stHig^{ss}(C',G,\omega_{C'})$ to a morphism $\widetilde{f}: \Theta_R \to \stHig^{ss}(C',G,\omega_{C'})$ such that we have $\widetilde{f}|_{\Theta_s}(1) = j(s)$ at the special fiber. By local constancy of the weight of a line bundle, we also have $wt((\widetilde{f}|_{\Theta_s})^*(\cL)|_0) =0$, thus showing that $j(s)$ is not stable.

    Constructibility follows from a standard argument. A semistable geometric point $(E,\phi)$ defined over some field extension $K \supset k$ is unstable if and only if it admits a non-central compatible weighted parabolic reduction $(\lambda, E_P)$ such that the corresponding morphism $f: \Theta_K \to \stHig(C',G,\omega_{C'})$ satisfies $wt(f^*(\cL)|_0) =0$. Again by \cite[Lem. 6.15]{alper2018existence}, this is equivalent to $f(0)$ being a semistable point in $\stHig^{ss}(C',G,\omega_{C'})$. Recall that $f(0)$ corresponds to the $G$-Higgs bundle associated to the $P$-Higgs bundle $(E_{P}, \phi_P)$ via the homomorphism $\psi: P \twoheadrightarrow L_{\lambda} \hookrightarrow G$. So a semistable geometric point is not stable if and only if it comes from a $P$-Higgs bundle for some parabolic subgroup $P \subsetneq G$ such that the associated Levi bundle (viewed as a $G$-Higgs bundle via a choice of the Levi subgroup splitting) is also semistable. To summarize this discussion, consider the open substack $\stHig^{ss}(C', P, \omega_{C'}) \subset \stHig(C', P, \omega_{C'})$ defined by the fiber product
    $\stHig(P)\times_{\stHig(G)}\stHig^{ss}$.
    If we take the union $\bigsqcup_{P} \stHig^{ss}(C', P, \omega_{C'})$ as we run over the finitely many conjugacy classes of parabolic subgroups $P \subsetneq G$ with some choice of Levi splitting, then the locus of geometric points in $\stHig^{ss}(C', G, \omega_{C'})$ that are not stable is exactly the image of $\bigsqcup_{P} \stHig^{ss}(C', P, \omega_{C'}) \to \stHig^{ss}(C', G, \omega_{C'})$. To conclude constructibility, by Chevalley's theorem it suffices to show that $\psi_*: \stHig^{ss}(C', P, \omega_{C'}) \to \stHig^{ss}(C', G, \omega_{C'})$ is of finite type. This is true because $\psi_*: \stHig(C', P, \omega_{C'}) \to \stHig(C', G, \omega_{C'})$ fits into the commutative diagram
    \[
        \xymatrix{
        \stHig(C', P, \omega_{C'}) \ar[d]^{\psi_*} \ar[r] & \Bun_{P}(C')  \ar[d]^{\psi_*} \\ \stHig(C', G, \omega_{C'}) \ar[r]
        & \Bun_{G}(C')
        },
    \]
    where the horizontal arrows are affine and of finite type, and the left vertical arrow is of finite type by \cite[Prop. 2.3 + Prop. 2.4(iii)]{herrero2020quasicompactness} (note that the statement of \cite[Prop. 2.4]{herrero2020quasicompactness} assumes that the characteristic of $k$ is zero, but this is only used in the proof to ensure that the unipotent radical $U$ is an extension of vector space groups where the source group $P$ acts linearly, which is satisfied in our case of a parabolic subgroup $P$ under the assumption $p\geq h(G)$ by the existence of an equivariant exponential map as in \cite[Prop. 5.3]{seitz-unipotent}).
\end{proof}

\begin{prop}
\label{lemma: kostant stable}
  If $p=0$ or $p\geq h(G)$, then the Kostant section $\kappa: A(G, \omega_{C})\to\stHig(C,G,\omega_{C})$ lands in the locus of stable geometric points inside $\stHig^{ss}(C,G, \omega_{C}).$
\end{prop}
\begin{proof}
  Because the Kostant section is $\bG_m$-equivariant, the stable locus is $\bG_m$-stable by its definition, and the stable locus is open by \Cref{lemma: openness of stability}, it suffices to show that $\kappa$ sends the origin of the Hitchin base $0_A$ to the stable locus. Let $(E,\phi) \in \stHig(C,G, \omega_{C})$ be the image $\kappa(0)$ of zero under the Kostant section. There is a description of $(E,\phi)$ as follows (see \cite[\S7.1]{dalakov2017lectures}). Since $p\geq h(G)$ or $p=0$, the regular nilpotent image $\kappa(0)$ of the Kostant section in $\fg$ determines a principal $SL_2$-group $\psi: SL_2 \to G$ \cite[Prop. 2]{serre1996exemples}. If we denote by $(E_{SL_2}, \phi_{SL_2})$ the zero Kostant section for $SL_2$, then we have $(E,\phi) = (\psi_*(E_{SL_2}), \psi_*(\phi))$. 
  
  Consider the quotient $q: G \to \overline{G}= G/Z_G$. Then the composition $q \circ \psi: SL_2 \to G \to \overline{G}$ yields a principal $SL_2$-group for $\overline{G}$. Hence, we may choose Kostant section for $\overline{G}$ so that $(q_*(E), q_*(\phi))$ is the zero Kostant section. By \cite[Prop. 2.23(b)]{hererro-zhang}, which establishes a correspondence between compatible parabolic reductions of $(E, \phi)$ and $\left(q_*(E), q_*(\phi)\right)$, we are reduced to the case when $G = \overline{G}$ is adjoint. 
  There is a decomposition $G = \prod_i G_i$, where $G_i$ are simple reductive groups. Each composition $q_i \circ \psi: SL_2 \to G \to G_i$ yields a principal $SL_2$-group for $G_i$, and hence $\left((q_i)_*(E), (q_i)_*(\phi)\right)$ is a zero Kostant section for $G_i$. Since stability can be checked on each factor, we are reduced to the case when $G$ is a simple reductive group; we shall impose this hypothesis for the rest of this proof. We consider two cases:

  \noindent $\bullet$ \underline{Case 1: $G$ is not of type $D_{2n}$.}  Consider the adjoint Higgs bundle $(ad(E),ad(\phi))$. As explained in \cite[\S7.1]{dalakov2017lectures}, there is a decomposition $ad(E)\cong \bigoplus_{m=-(h(G)-1)}^{h(G)-1} \fg_m\otimes \omega_{C}^{\otimes m/2},$ where $\fg=\bigoplus_{m=-(h(G)-1)}^{h(G)-1}\fg_m$ is the grading by weight of $\fg$ as an $SL_2$-representation via the associated Jacobson-Morozov triple $\psi: SL_2 \to G$. The Higgs field $ad(\phi)$ acts on $\bigoplus_{m=-(h(G)-1)}^{h(G)-1} \fg_m\otimes \omega_{C}^{\otimes m/2}$ as an $\omega_{C}$-twisted version of the lowering operator corresponding to the lowering nilpotent element in $\mathfrak{sl}_2$. 
  Since $p\geq h(G)$ or $p=0$, we may decompose $\fg$ as a direct sum $\fg = \bigoplus_l V^l$ of irreducible $SL_2$-representations \cite[Thm 6, pg. 25]{serre20031998}. It is proven in \cite{kostant-3d} that the number of such irreducible representations equals the rank $r$ of the simple Lie algebra $\fg$, and the highest weights of the summands $V^l$ are determined by the exponents of the Lie algebra as in \cite[Chpt.6, \S4.5-\S4.13]{bourbaki-lie}.
  For every simple group except type $D_{2n}$, there are $r$ distinct exponents. It follows that all the irreducible representations arising in the decomposition $\fg = \bigoplus_l V^l$ are pairwise non-isomorphic, and so the decomposition is canonical. This induces a decomposition of the Higgs bundle $ad(E) = \bigoplus_{l} ad(E)^l$ with $ad(E)^l:= \bigoplus_{-M_l \leq m \leq M_l} V^l_m \otimes \omega_{C}^{\otimes m/2}$, where each weight space $V^l_m$ has dimension $1$. The lowering operator $ad(\phi)$ preserves each summand $ad(E)^l$. For each tuple of number $n_l$ with $-M_l \leq n_l \leq M_l$, there is a subbundle preserved by $ad(\phi)$
    \[ ad(E)^l_{n_l} := \bigoplus_{-M_l \leq m \leq n_l} V^l_m \otimes \omega_{C}^{\otimes m/2} \subset ad(E)^l\]
    Furthermore, our description above implies that every $ad(\phi)$-preserved subbundle of $ad(E)$ is of the form $\bigoplus_l ad(E)^l_{n_l}$ for some tuple of integers $n_l$. Using $\text{deg}(\omega_{C})>0$ (by our assumption that $g(C)\geq 2$) and the description of $ad(\phi)$ invariant subbundles, it follows that each direct summand $ad(E)^l$ is stable as a vector Higgs bundle (recall that a vector bundle with a Higgs field $(E,\phi)$ is stable if for any $\phi$-invariant nontrivial proper subbundle $F$ of $E,$ we have the slope inequality $\mu(F)<\mu(E)$). In particular, $(ad(E), ad(\phi))$ is a semistable Higgs bundle.

    Now let us study stability of the original $(E,\phi)$. Choose a weighted parabolic reduction $(\lambda, E_P)$ compatible with $\phi$ (as in \cite[\S4.1]{hererro-zhang}). This induces a $\mathbb{Z}$-weighted filtration $(ad(E)_{i})_{i \in \mathbb{Z}}$ by subbundles $ad(E)_i \subset ad(E)$ that are preserved by the Higgs field $ad(\phi)$. Indeed, we set $ad(E)_i := E_{P}(\fg_{\lambda \geq i})$, where $\fg_{\lambda \geq i} \subset \fg$ is the $P$-subrepresentation where $\lambda$ acts with weight at least $i$. 
    The interpretation of stability in the proof of \Cref{lemma: action preserves stability} and the computation of the weight of the determinant line bundle in \cite[\S1.F.c]{heinloth-hilbertmumford} jointly imply that $(E, \phi)$ is stable if and only if whenever $\lambda$ is not central, we have
    \[ \sum_{i} i \cdot \text{deg}(ad(E)_i/ad(E)_{i+1}) = \sum_{i} i \cdot \left[\text{deg}(ad(E)_i) - \text{deg}(ad(E)_{i+1})\right] <0.\]
    Using the additivity of degree, summation by parts, and the fact that $\text{deg}(E) =0$, we get
    \[\sum_{i} i \cdot \left[\text{deg}(ad(E)_i) - \text{deg}(ad(E)_{i+1})\right] = \sum_{i} \text{deg}(ad(E)_i)\]
    By semistability of $(ad(E), ad(\phi))$, it follows that $\text{deg}(ad(E)_i) \leq 0$ for all $i$. Therefore, in order to show the inequality $\sum_{i} \text{deg}(ad(E)_i)<0$ it suffices to prove that there exists some $i$ such that $\text{deg}(ad(E)_i)<0$. We shall show indeed that $\text{deg}(ad(E)_1) <0$. Suppose for the sake of contradiction that $\text{deg}(ad(E)_1)=0$. By our description of subbundles preserved by $ad(\phi)$, this implies that $ad(E)_1$ must be a direct sum of some of the stable Higgs bundle summands $ad(E)_1 = \bigoplus_j ad(E)^{l_j}$ coming from the decomposition $\mathfrak{g} = \bigoplus_l V^l$.
    
    Since we assume $p\nmid |W|$ throughout the paper, there exists a choice of nondegenerate symmetric bilinear pairing $b: \mathfrak{g} \otimes \mathfrak{g} \to k$, as $G$-representations (where $G$ acts trivially on $k$) \cite[Lem. 4.2.3]{riche-kostant-section}. This is also a symmetric nondegenerate pairing of $SL_2$-representations. By using Schur's lemma and self-duality of $SL_2$-representations, it follows that the canonical decomposition into isotypic $SL_2$-components $\mathfrak{g} = V^l$ satisfies that each restriction $b: V^l \otimes V^l \to k$ remains nondegenerate.

    After twisting by the $G$-bundle $E$, the pairing $b$ induces a nondegenerate symmetric pairing $b: ad(E) \otimes ad(E) \to \cO_{C}$. Furthermore, it restricts to a nondegenerate pairing $ad(E)^l \otimes ad(E)^l \to \cO_{C}$ for each $l$. Since $ad(E)_1 = \bigoplus_j ad(E)^{l_j}$, it follows that the restriction $ad(E)_1 \otimes ad(E)_1 \to \cO_{C}$ is nondegenerate. By the construction of the filtration, we have that $ad(E)_1 = E_{P}(\mathfrak{g}_{\lambda \geq 1})$. Since the cocharacter $\lambda$ is not in the center, it follows that $\fg_{\lambda \geq 1} \neq 0$, which means that $ad(E)_1 \neq 0$. On the other hand, it follows from $\lambda$-weight considerations that the restriction pairing $b:\fg_{\lambda \geq 1}\otimes \fg_{\lambda \geq 1} \to k$ of $P$-representations is identically $0$. By twisting by the $P$-bundle $E_{P}$, we conclude that the restriction of the pairing $ad(E)_1 \otimes ad(E)_1 \to \cO_{C}$ is identically zero, a contradiction.

     \noindent$\bullet$ \underline{Case 2: $G$ is of type $D_{2n}$.} We shall modify slightly the argument in Case 1. We keep the same notation as in the previous case. 
     
     For type $D_{2n}$, one of the exponents of the Lie algebra (namely, $2n-1$) is repeated twice. Therefore, the decomposition of $\fg$ into isotypic $SL_2$-components is of the form $\fg = \left(\bigoplus_{l \neq \sigma} V^l \right) \oplus (V^{\sigma})^{\oplus 2}$, where the $V^l$ are distinct irreducible representations. The irreducible representation $V^{\sigma}$ has dimension $4n-1$. Schur's lemma shows that the bilinear pairing $b: \fg \otimes \fg \to k$ remains nondegenerate when restricted to each $V^l$ for $l \neq \sigma$. It also remains nondegenerate when restricted to the isotypic $SL_2$-component $(V^{\sigma})^{\oplus 2}$. It follows that any nonzero $b$-isotropic $SL_2$-subrepresentation of $\fg$ must be isomorphic to $V^{\sigma}$, and so it has dimension $4n-1$.

     Choose a $\phi$-compatible weighted parabolic reduction $(\lambda,E_P)$. Just as in the argument in Case 1, it suffices to show that $\mathrm{deg}(ad(E)_1)<0$. Assume for the sake of contradiction that $\mathrm{deg}(ad(E)_1)=0$. The same considerations as in Case 1 show that the subbundle $ad(E)_1 \subset ad(E)$ is of the form $E_{{SL_2}}(V) \subset E_{SL_2}(\fg) = ad(E)$ for some $SL_2$-subrepresentation $V \subset \fg$ 
     (indeed, $ad(E)_1$ is still a combination of stable Higgs bundle direct summands of $(ad(E), ad(\phi))$ coming from irreducible $SL_2$-subrepresentations of $\fg$). Furthermore, $ad(E)_1$ is again nonzero and $b$-isotropic, and hence it must come from a nonzero $b$-isotropic $SL_2$-subrepresentation $V \subset \fg$. This means that its rank must be $4n-1$. On the other hand, the rank of $ad(E)_1 = E_P(\fg_{\lambda \geq 1})$ agrees with the dimension of $\fg_{\lambda \geq 1}$, which is the Lie algebra of the unipotent radical of the parabolic subgroup $P$. One may compute the dimensions of unipotent radicals of maximal parabolic subgroups in type $D_{2n}$ by considering all possible ways of removing a node from the Dynkin diagram. This way it can be checked that all parabolic subgroups $P \subsetneq G$ have unipotent radicals of dimension larger than $4n-1$, a contradiction.
\end{proof}

\end{section}

\bibliographystyle{plain}
\footnotesize{\bibliography{coh_G_bundles.bib}}

\end{document}